\documentclass[12pt]{article}

\usepackage{amsmath}
\usepackage{latexsym}
\usepackage{amssymb}
\usepackage{ifthen}
\usepackage{tikz}
\usepackage{amsthm}
\usepackage{yfonts}

\usepackage{hyperref}

\newcommand{\junk}[1]{}

\newcommand{\dword}[1]{\textit{\bf #1}}

\newcommand{\ms}{\mathbf{cfms}}
\newcommand{\ded}{\ensuremath{\mathbf{d}}}
\newcommand{\zf}{\ensuremath{\mathbf{czf}}}
\newcommand{\mm}{\ensuremath{\mathbf{mu}}}
\newcommand{\degree}{\mathrm{deg}}

\newtheorem{defn}{Definition}[section]
\newtheorem{thm}[defn]{Theorem}

\newtheorem{cor}[defn]{Corollary}

\newtheorem{lemma}[defn]{Lemma}

\newtheorem{quest}[defn]{Question}

\title {Deduction, Constrained Zero Forcing, and Constrained Searching}
\author{Andrea Burgess\thanks{Department of Mathematics and Statistics, University of New Brunswick, Saint John, NB, E2L 4L5, Canada.  Supported by NSERC Discovery Grant RGPIN-2019-04328.}, Danny Dyer\thanks{Department of Mathematics and Statistics, Memorial University of Newfoundland, St.~John's, NL, A1C 5S7, Canada. Supported by NSERC Discovery Grant RGPIN-2021-03064.}, 
Kerry Ojakian\thanks{Department of Mathematics and Computer Science, Bronx Community College of The City University of New York (Bronx, NY, U.S.A)}, 
\\
Lusheng Wang\thanks{Department of Computer Science, City University of Hong Kong, Hong Kong Special Administrative Region, China. Supported by National Science Foundation of China (NSFC: 61972329) and GRF grants for Hong Kong Special Administrative Region, P.R. China (CityU 11206120 and CityU 11218821).}, Mingyu Xiao\thanks{School of Computer Science and Engineering, University of Electronic Science and Technology of China, Chengdu, China. Supported by National Natural Science Foundation of China (NSFC: 62372095)}, Boting Yang\thanks{Department of Computer Science, University of Regina, Regina, SK, S4S 0A2, Canada. Supported by NSERC Discovery Research Grant RGPIN-2018-06800.}}

\begin{document}
\maketitle

\begin{abstract}
Deduction is a recently introduced graph searching process in which searchers clear the vertex set of a graph with one move each, with each searcher's movement determined by which of its neighbors are protected by other searchers.  In this paper, we show that the minimum number of searchers required to clear the graph is the same in deduction as in constrained versions of other previously studied graph processes, namely zero forcing and fast-mixed search.  We give a structural characterization, new bounds and a spectrum result on the number of searchers required.  We consider the complexity of computing this parameter, giving an NP-completeness result for arbitrary graphs, and exhibiting families of graphs for which the parameter can be computed in polynomial time. We also describe properties of the deduction process related to the timing of searcher movement and the success of terminal layouts.
\end{abstract}

\section{Introduction}

We introduce three distinct graph searching parameters, and one property concerning graph structure.  We show the parameters are in fact equivalent.  Then we derive various properties of this common parameter. 
Roughly, the three search parameters all involve placing ``searchers'' on vertices of the graph, and then giving them all the ability to move once, subject to various constraints.  We now define the three search parameters precisely; we will see that they are essentially different ways to conceive of the same process.

We first define a constrained version of zero forcing.  Standard terminology in zero forcing starts with a graph having some vertices colored, and the rest of the vertices uncolored.  A colored vertex $c$ can \dword{force} an uncolored vertex $u$ to become colored, if the only uncolored neighbor of $c$ is $u$.  In the standard version of zero forcing, some initial vertices start colored, and the forcing rule is repeatedly applied, in order to include more vertices among the colored vertices. 
A set of initially colored vertices which force all vertices of the graph to become colored is called a \dword{zero forcing set}; the \dword{zero forcing number} of a graph is the minimum size of a zero forcing set.  The zero forcing number was introduced in~\cite{AIM}, and has been extensively studied; see the monograph~\cite{ZFMonograph} for details on zero forcing and its applications, among which are linear algebra \cite{linearalgebra} and quantum physics \cite{quantumexample}. Zero-forcing has connections to graph searching processes, for instance, in~\cite{EMP} it is shown that for any graph without isolated vertices, the zero forcing number of its line graph is an upper bound on the graph's brushing number.  

We define \dword{constrained zero forcing} to refer to the same zero forcing process, but where only the initially colored vertices can force a vertex to become colored.
Given a graph $G$, let $\zf(G)$ be the minimum size of an initial colored set which leads to all the vertices eventually being colored.

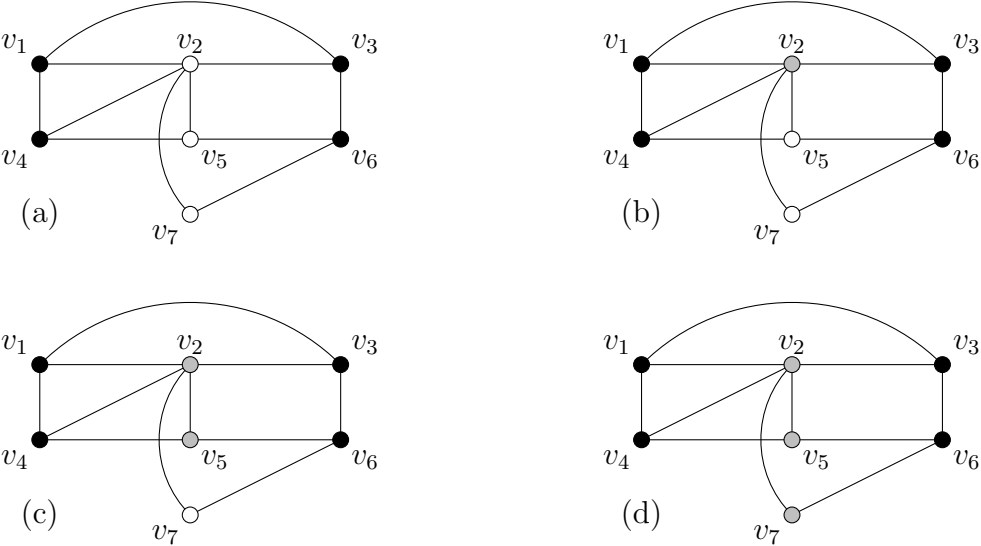
\begin{figure}[htb]
	\centering
	\begin{tikzpicture}[x=1cm,y=1cm,scale=1]

\newcommand*{\xshift}{8}%
\newcommand*{\yshift}{4}%

\foreach \x in {0,\xshift}
\foreach \y in {0,\yshift}
{
\draw (\x,\y+2) -- (\x + 2,\y+2) -- (\x +4,\y+2);
\draw (\x +0,\y+1) -- (\x +2,\y+1) -- (\x +4,\y+1);
\draw (\x +0,\y+1) -- (\x +0,\y+2);
\draw (\x +2,\y+1) -- (\x +2,\y+2);
\draw (\x +4,\y+1) -- (\x +4,\y+2);
\draw (\x +0,\y+1) -- (\x +2,\y+2);
\draw (\x +2,\y+0) -- (\x +4,\y+1);
\draw (\x +2,\y+0) to [out=135, in=225] (\x +2,\y+2); 
\draw (\x +0,\y+2) to [out=45, in=135] (\x +4,\y+2); 

\draw[fill=black] (\x +0,\y+2) circle (3pt) node[above left]{$v_1$};
\draw[fill=black] (\x +4,\y+2) circle (3pt) node[above right]{$v_3$};

\draw[fill=black] (\x +0,\y+1) circle (3pt) node[below left]{$v_4$};
\draw[fill=black] (\x +4,\y+1) circle (3pt) node[below right]{$v_6$};
}

\draw[fill=white] (2,2+\yshift) circle (3pt) node[above ]{$v_2$};
\draw[fill=white] (2,1+\yshift) circle (3pt) node[below right]{$v_5$};
\draw[fill=white] (2,0+\yshift) circle (3pt) node[below left]{$v_7$};

\draw[fill=lightgray] (2+\xshift,2+\yshift) circle (3pt) node[above ]{$v_2$};
\draw[fill=white] (2+\xshift,1+\yshift) circle (3pt) node[below right]{$v_5$};
\draw[fill=white] (2+\xshift,0+\yshift) circle (3pt) node[below left]{$v_7$};

\draw[fill=lightgray] (2,2) circle (3pt) node[above ]{$v_2$};
\draw[fill=lightgray] (2,1) circle (3pt) node[below right]{$v_5$};
\draw[fill=white] (2,0) circle (3pt) node[below left]{$v_7$};

\draw[fill=lightgray] (2+\xshift,2) circle (3pt) node[above ]{$v_2$};
\draw[fill=lightgray] (2+\xshift,1) circle (3pt) node[below right]{$v_5$};
\draw[fill=lightgray] (2+\xshift,0) circle (3pt) node[below left]{$v_7$};

\draw (0,0) node{(c)};
\draw (\xshift,0) node{(d)};
\draw (\xshift,\yshift) node{(b)};
\draw (0,\yshift) node{(a)};

	\end{tikzpicture}
	\caption{A graph $G$ with $\zf(G) = 4$, shown (a) initially; (b) after one vertex is forced; (c) after two vertices are forced; and (d) at the end of the process.} \label{fig:examplezf}
\end{figure}

Consider the graph $G$ in Figure~\ref{fig:examplezf}(a), with constrained zero forcing set $\{v_1, v_3, v_4, v_6\}$, colored black. We see that $v_2$ is the only vertex adjacent to $v_1$ that is uncolored, and hence $v_1$ forces $v_2$ to be colored (gray), as in Figure~\ref{fig:examplezf}(b). After this forcing, we see that $v_4$ now has only one uncolored adjacent vertex, $v_5$, which it forces to be colored (Figure~\ref{fig:examplezf}(c)). Now, vertex $v_7$ is the only uncolored vertex adjacent to $v_2$. However, $v_2$ cannot force $v_7$, as $v_2$ is not one of the initially colored black vertices. However, $v_6$ can force $v_7$, and hence we end with all vertices colored, as required. Thus, $\zf(G) \le 4$, and the result that $\zf(G)=4$ follows by exhaustively checking all possible $3$-vertex sets.

We next define a constrained version of \dword{fast-mixed search}.  Standard terminology in mixed search starts with a graph, all of whose vertices and edges are initially \dword{contaminated}.  
In fast-mixed search, introduced in~\cite{Yang13}, searchers attempt to decontaminate the graph.  
In each step, we may 
\dword{place} a searcher on a contaminated vertex or
\dword{slide} a searcher along an edge $uv$ from vertex $u$ to vertex $v$ if $uv$ is the only contaminated edge incident with $u$.
An edge $e$ is considered \dword{cleared} if either of the following conditions are satisfied:
\begin{enumerate}  
\item Each endpoint of $e$ is occupied by a searcher, or 
\item A searcher slides along $e$. 
\end{enumerate}

Note that rule 2 is only needed when there is a single searcher at $u$, so that
when that searcher moves to an adjacent vertex $v$, rule 1 does not apply.  In the mixed search model of Bienstock and Seymour \cite{mixedsearch}, searchers can be removed from a vertex and subsequently placed on another vertex.  By contrast, fast-mixed search allows only the placing and sliding actions.  In~\cite{FMY14}, it is shown that the minimum number of searchers required to clear a graph in fast-mixed search is equal to its zero forcing number.
In \dword{constrained fast-mixed search}, we restrict each searcher to sliding at most once.
Given a graph $G$, let $\ms(G)$ be the least number of searchers required so that all the edges are cleared.

In the definition of (constrained) fast-mixed search, searchers may be placed at any stage of the process, allowing an interleaving of the placing and sliding actions.  However, for our purposes we will use the convention that all searchers are placed at the beginning, before any sliding actions take place; such a strategy is called \dword{normalized}.  Note that the assumption of a normalized strategy does not affect the number of searchers required to clear the graph.

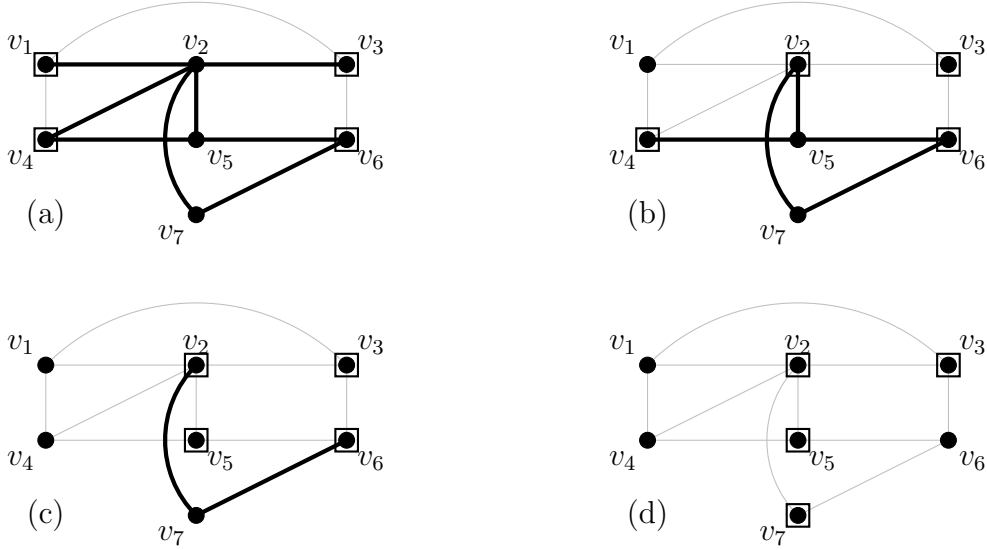
\begin{figure}[htb]
	\centering
	\begin{tikzpicture}[x=1cm,y=1cm,scale=1]
	
	\newcommand*{\xshift}{8}%
	\newcommand*{\yshift}{4}%

\draw [color=lightgray](\xshift,\yshift+2) -- (\xshift + 2,\yshift+2)-- (\xshift +4,\yshift+2);
\draw [ultra thick](\xshift +0,\yshift+1) -- (\xshift +2,\yshift+1) -- (\xshift +4,\yshift+1);
\draw [color=lightgray](\xshift +0,\yshift+1) -- (\xshift +0,\yshift+2);
\draw [ultra thick](\xshift +2,\yshift+1) -- (\xshift +2,\yshift+2);
\draw [color=lightgray](\xshift +4,\yshift+1) -- (\xshift +4,\yshift+2);
\draw [color=lightgray](\xshift +0,\yshift+1) -- (\xshift +2,\yshift+2);
\draw [ultra thick](\xshift +2,\yshift+0) -- (\xshift +4,\yshift+1);
\draw [ultra thick] (\xshift +2,\yshift+0) to [out=135, in=225] (\xshift +2,\yshift+2); 
\draw [color=lightgray](\xshift +0,\yshift+2) to [out=45, in=135] (\xshift +4,\yshift+2);

\draw [ultra thick](0,\yshift+2) -- ( 2,\yshift+2)-- (4,\yshift+2);
\draw [ultra thick](0,\yshift+1) -- (2,\yshift+1) -- (4,\yshift+1);
\draw [color=lightgray](0,\yshift+1) -- (0,\yshift+2);
\draw [ultra thick](2,\yshift+1) -- (2,\yshift+2);
\draw [color=lightgray](4,\yshift+1) -- (4,\yshift+2);
\draw [ultra thick](0,\yshift+1) -- (2,\yshift+2);
\draw [ultra thick](2,\yshift+0) -- (4,\yshift+1);
\draw [ultra thick] (2,\yshift+0) to [out=135, in=225] (2,\yshift+2); 
\draw [color=lightgray](0,\yshift+2) to [out=45, in=135] (4,\yshift+2); 

\draw [color=lightgray](0,2) -- ( 2,2)-- (4,2);
\draw [color=lightgray](0,1) -- (2,1) -- (4,1);
\draw [color=lightgray](0,1) -- (0,2);
\draw [color=lightgray](2,1) -- (2,2);
\draw [color=lightgray](4,1) -- (4,2);
\draw [color=lightgray](0,1) -- (2,2);
\draw [ultra thick](2,0) -- (4,1);
\draw [ultra thick] (2,0) to [out=135, in=225] (2,2); 
\draw [color=lightgray](0,2) to [out=45, in=135] (4,2); 

\draw [color=lightgray](\xshift,2) -- (\xshift + 2,2)-- (\xshift +4,2);
\draw [color=lightgray](\xshift +0,1) -- (\xshift +2,1) -- (\xshift +4,1);
\draw [color=lightgray](\xshift +0,1) -- (\xshift +0,2);
\draw [color=lightgray](\xshift +2,1) -- (\xshift +2,+2);
\draw [color=lightgray](\xshift +4,1) -- (\xshift +4,2);
\draw [color=lightgray](\xshift +0,1) -- (\xshift +2,2);
\draw [color=lightgray](\xshift +2,0) -- (\xshift +4,1);
\draw [color=lightgray](\xshift +2,0) to [out=135, in=225] (\xshift +2,2); 
\draw [color=lightgray](\xshift +0,2) to [out=45, in=135] (\xshift +4,2);

	\foreach \x in {0,\xshift}
\foreach \y in {0,\yshift}
{
	
	\draw[fill=black] (\x +0,\y+2) circle (3pt) node[above left]{$v_1$};
	\draw[fill=black] (\x +4,\y+2) circle (3pt) node[above right]{$v_3$};
	
	\draw[fill=black] (\x +0,\y+1) circle (3pt) node[below left]{$v_4$};
	\draw[fill=black] (\x +4,\y+1) circle (3pt) node[below right]{$v_6$};
	\draw[fill=black] (2+\x,2+\y) circle (3pt) node[above ]{$v_2$};
	\draw[fill=black] (2+\x,1+\y) circle (3pt) node[below right]{$v_5$};
	\draw[fill=black] (2+\x,0+\y) circle (3pt) node[below left]{$v_7$};

}

\newcommand*{\smallshift}{0.15}%

\draw [thick](2+\xshift-\smallshift,2+\yshift-\smallshift) rectangle (2+\xshift+\smallshift,2+\yshift+\smallshift);
\draw [thick](4+\xshift-\smallshift,2+\yshift-\smallshift) rectangle (4+\xshift+\smallshift,2+\yshift+\smallshift);
\draw [thick](0+\xshift-\smallshift,1+\yshift-\smallshift) rectangle (0+\xshift+\smallshift,1+\yshift+\smallshift);
\draw [thick](4+\xshift-\smallshift,1+\yshift-\smallshift) rectangle (4+\xshift+\smallshift,1+\yshift+\smallshift);

\draw [thick](2+\xshift-\smallshift,2-\smallshift) rectangle (2+\xshift+\smallshift,2+\smallshift);
\draw [thick](4+\xshift-\smallshift,2-\smallshift) rectangle (4+\xshift+\smallshift,2+\smallshift);
\draw [thick](2+\xshift-\smallshift,1-\smallshift) rectangle (2+\xshift+\smallshift,1+\smallshift);
\draw [thick](2+\xshift-\smallshift,0-\smallshift) rectangle (2+\xshift+\smallshift,0+\smallshift);

\draw [thick](0-\smallshift,2+\yshift-\smallshift) rectangle (0+\smallshift,2+\yshift+\smallshift);
\draw [thick](4-\smallshift,2+\yshift-\smallshift) rectangle (4+\smallshift,2+\yshift+\smallshift);
\draw [thick](0-\smallshift,1+\yshift-\smallshift) rectangle (0+\smallshift,1+\yshift+\smallshift);
\draw [thick](4-\smallshift,1+\yshift-\smallshift) rectangle (4+\smallshift,1+\yshift+\smallshift);

\draw [thick](2-\smallshift,2-\smallshift) rectangle (2+\smallshift,2+\smallshift);
\draw [thick](4-\smallshift,2-\smallshift) rectangle (4+\smallshift,2+\smallshift);
\draw [thick](2-\smallshift,1-\smallshift) rectangle (2+\smallshift,1+\smallshift);
\draw [thick](4-\smallshift,1-\smallshift) rectangle (4+\smallshift,1+\smallshift);

	\draw (0,0) node{(c)};
	\draw (\xshift,0) node{(d)};
	\draw (\xshift,\yshift) node{(b)};
	\draw (0,\yshift) node{(a)};

	\end{tikzpicture}
	\caption{A graph $G$ with $\ms(G) = 4$, shown (a) initially; (b) after sliding from $v_1$ to $v_2$; (c) after sliding from $v_4$ to $v_5$; and (d) after sliding from $v_6$ to $v_7$.} \label{fig:examplefms}
\end{figure}

In Figure~\ref{fig:examplefms}, we consider the same graph as in our previous example, but in the constrained fast-mixed search model. In this case, in Figure~\ref{fig:examplefms}(a), we see that four searchers have been placed on the vertices $v_1$, $v_3$, $v_4$, and $v_6$, indicated by boxes around the vertices. All edges with both endpoints occupied are then cleared (in gray) and the other edges are contaminated (in black). The searcher on $v_1$ is incident with only one contaminated edge, $v_1v_2$, so slides to $v_2$, clearing this edge. This then clears edges $v_2v_4$ and $v_2v_3$, as in Figure~\ref{fig:examplefms}(b). At this point, only the searcher on $v_4$ may move, sliding along and clearing $v_4v_5$. This subsequently also clears $v_5v_2$ and $v_5v_6$, as in Figure~\ref{fig:examplefms}(c). At this point, only the searcher on $v_6$ may move (as the searcher on $v_2$ has already moved). This searcher slides along $v_6v_7$, clearing that edge, and then subsequently clearing $v_7v_2$, which gives a completely uncontaminated graph. It is then straightforward to argue that $\ms(G) = 4$. 

We recall a recent definition of a search procedure called \dword{deduction}, introduced in~\cite{BDF,thesis}. 
In this procedure, the searchers 
place themselves at some vertices of the graph, those vertices being considered \dword{protected}, and the rest of the vertices are considered \dword{unprotected}.  The searchers then move in a series of \dword{stages}: 
1) at a stage, consider all the vertices that contain searchers who have not yet moved;
2) if the number of searchers at such a vertex is at least the number of unprotected neighbors,
then move one searcher to each of the unprotected neighbors; any excess searchers on a vertex may be moved to an arbitrary neighbor or left in place.  The vertices that the searchers moved from remain protected, and the vertices moved to, called their \dword{targets}, are also considered protected.
The procedure ends when no vertices with unprotected neighbors meet the conditions in (2).
Given a graph $G$, let $\ded(G)$ be the least number of searchers required so that the deduction procedure can start with
that many searchers and protect the entire vertex set.

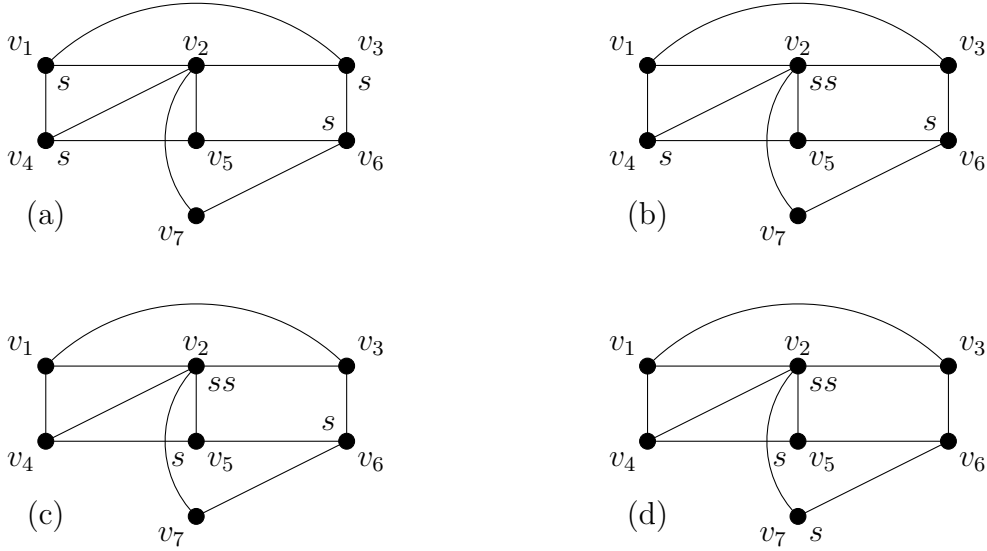
\begin{figure}[htb]
	\centering
	\begin{tikzpicture}[x=1cm,y=1cm,scale=1]
	
	\newcommand*{\xshift}{8}%
	\newcommand*{\yshift}{4}%
	
	\foreach \x in {0,\xshift}
	\foreach \y in {0,\yshift}
	{
		\draw (\x,\y+2) -- (\x + 2,\y+2) -- (\x +4,\y+2);
		\draw (\x +0,\y+1) -- (\x +2,\y+1) -- (\x +4,\y+1);
		\draw (\x +0,\y+1) -- (\x +0,\y+2);
		\draw (\x +2,\y+1) -- (\x +2,\y+2);
		\draw (\x +4,\y+1) -- (\x +4,\y+2);
		\draw (\x +0,\y+1) -- (\x +2,\y+2);
		\draw (\x +2,\y+0) -- (\x +4,\y+1);
		\draw (\x +2,\y+0) to [out=135, in=225] (\x +2,\y+2); 
		\draw (\x +0,\y+2) to [out=45, in=135] (\x +4,\y+2); 
		
		\draw[fill=black] (\x +0,\y+2) circle (3pt) node[above left]{$v_1$};
		\draw[fill=black] (\x +4,\y+2) circle (3pt) node[above right]{$v_3$};
		
		\draw[fill=black] (\x +0,\y+1) circle (3pt) node[below left]{$v_4$};
		\draw[fill=black] (\x +4,\y+1) circle (3pt) node[below right]{$v_6$};

		\draw[fill=black] (2+\x,2+\y) circle (3pt) node[above ]{$v_2$};
		\draw[fill=black] (2+\x,1+\y) circle (3pt) node[below right]{$v_5$};
		\draw[fill=black] (2+\x,0+\y) circle (3pt) node[below left]{$v_7$};

	}
	
	\draw (2+\xshift,2+\yshift) node[below right]{$ss$};
	\draw (0+\xshift,1+\yshift) node[below right]{$s$};
	\draw (4+\xshift,1+\yshift) node[above left]{$s$};
	\draw (0,2+\yshift) node[below right]{$s$};
	\draw (4,2+\yshift) node[below right]{$s$};
	\draw (0,1+\yshift) node[below right]{$s$};
	\draw (4,1+\yshift) node[above left]{$s$};
\draw (2,2) node[below right]{$ss$};
\draw (2,1) node[below left]{$s$};
\draw (4,1) node[above left]{$s$};
\draw (2+\xshift,2) node[below right]{$ss$};
\draw (2+\xshift,1) node[below left]{$s$};
\draw (2+\xshift,0) node[below right]{$s$};

	\draw (0,0) node{(c)};
	\draw (\xshift,0) node{(d)};
	\draw (\xshift,\yshift) node{(b)};
	\draw (0,\yshift) node{(a)};

	\end{tikzpicture}
	\caption{A graph $G$ with $\ded(G) = 4$, shown (a) initially; (b) after the first stage; (c) after the second stage; and (d) at the end of the process.} \label{fig:exampleded}
\end{figure}

Again, returning to the same graph, we will show that $\ded(G) = 4$. Consider the searchers placed on the vertices $v_1$, $v_3$, $v_4$, and $v_6$, one per vertex, as in Figure~\ref{fig:exampleded}(a). Then these vertices are protected. In the first stage, we see that the vertices $v_1$ and $v_3$ both have only one unprotected neighbour, and contain (exactly) one searcher. Thus, these searchers both move to the unprotected neighbour $v_2$, protecting it, as in Figure~\ref{fig:exampleded}(b). In the second stage, the only vertex that contains searchers who have not yet moved that contains at least as many unmoved searchers as unprotected vertices is $v_4$, and hence that searcher moves to $v_5$, protecting that vertex, as in Figure~\ref{fig:exampleded}(c). In the third and final stage, the only unmoved searcher is on $v_6$, and since $v_6$ has only one unprotected neighbour ($v_7$), this searcher moves to $v_7$, protecting it, and giving us Figure~\ref{fig:exampleded}(d). Having shown that $\ded(G) \le 4$, it is straightforward to show that $\ded(G) = 4$.

Reflecting on the three graph searching parameters we can make some general observations.  They all involve a choice of initial vertices, which are colored, or contain searchers (possibly more than one per vertex); we call this initial choice a \dword{layout}.  Then, to carry out the process, we make a sequence of vertex choices from the vertices in the layout (choosing each at most once),
which satisfy particular properties depending on which of the three procedures we are considering. When we choose a vertex, this affects its neighboring vertices and incident edges in some ways; we refer to this generally as \dword{firing} the vertex. 
In particular, when using the deduction model, by a \dword{firing sequence} we mean a sequence of vertex sets $(S_1, S_2, \ldots, S_k)$, such that the vertices of $S_i$ are fired at stage $i$.  For example, in Figure~\ref{fig:exampleded}, the 
firing sequence is $(\{v_1, v_3 \}, \{ v_4 \}, \{ v_6 \})$. 
If from an initial layout, one of the three processes terminates by reaching all the vertices, we refer to that layout as \dword{successful}.

We now define a structural parameter of graphs which is useful in characterizing successful layouts.
For a graph $G$, a subset $M \subseteq E(G)$ is called a \dword{matching} if no two edges of $M$ share a common vertex.   A standard graph parameter is to find the size of a maximum matching in a graph.
We define a variant of this graph parameter.  We say that a matching $M$ in a graph has the \dword{uniqueness property}
if there exists a disjoint pair of vertex sets $V_1$ and $V_2$ so that $M$ is the unique perfect matching in the subgraph of $G$ formed by the edges from $V_1$ to $V_2$. 
We say that the maximal bipartite subgraph of $G$ with bipartition $(V_1,V_2)$ \dword{witnesses} the uniqueness property.
 Given a graph $G$, let
$\mm(G)$  be the maximum size of a matching in $G$ that has the uniqueness property.

\begin{figure}[htb]
	\centering
	\begin{tikzpicture}[x=1cm,y=1cm,scale=1]
	
	\newcommand*{\xshift}{8}%
	\newcommand*{\yshift}{4}%
	
	\foreach \x in {0,\xshift}
	\foreach \y in {0}
	{
		\draw (\x,\y+2) -- (\x + 2,\y+2) -- (\x +4,\y+2);
		\draw (\x +0,\y+1) -- (\x +2,\y+1) -- (\x +4,\y+1);
		\draw (\x +0,\y+1) -- (\x +0,\y+2);
		\draw (\x +2,\y+1) -- (\x +2,\y+2);
		\draw (\x +4,\y+1) -- (\x +4,\y+2);
		\draw (\x +0,\y+1) -- (\x +2,\y+2);
		\draw (\x +2,\y+0) -- (\x +4,\y+1);
		\draw (\x +2,\y+0) to [out=135, in=225] (\x +2,\y+2); 
		\draw (\x +0,\y+2) to [out=45, in=135] (\x +4,\y+2);

	}

\draw[line width = 3] (0,2) -- (2,2);
\draw[line width = 3] (0,1) -- (2,1);
\draw[line width = 3] (2,0) -- (4,1);

\draw[line width = 3] (4+\xshift,2) -- (2+\xshift,2);
\draw[line width = 3] (0+\xshift,1) -- (2+\xshift,1);
\draw[line width = 3] (2+\xshift,0) -- (4+\xshift,1);

\draw[fill=black] (0,2) circle (3pt) node[above left]{$v_1$};
\draw[fill=lightgray] (4,2) circle (3pt) node[above right]{$v_3$};

\draw[fill=black] (0,1) circle (3pt) node[below left]{$v_4$};
\draw[fill=black] (4,1) circle (3pt) node[below right]{$v_6$};

\draw[fill=white] (2,2) circle (3pt) node[above ]{$v_2$};
\draw[fill=white] (2,1) circle (3pt) node[below right]{$v_5$};
\draw[fill=white] (2,0) circle (3pt) node[below left]{$v_7$};
	
\draw[fill=lightgray] (\xshift +0,2) circle (3pt) node[above left]{$v_1$};
\draw[fill=black] (\xshift +4,2) circle (3pt) node[above right]{$v_3$};

\draw[fill=black] (\xshift +0,1) circle (3pt) node[below left]{$v_4$};
\draw[fill=black] (\xshift +4,1) circle (3pt) node[below right]{$v_6$};

\draw[fill=white] (2+\xshift,2) circle (3pt) node[above ]{$v_2$};
\draw[fill=white] (2+\xshift,1) circle (3pt) node[below right]{$v_5$};
\draw[fill=white] (2+\xshift,0) circle (3pt) node[below left]{$v_7$};

	\draw (0,0) node{(a)};
	\draw (\xshift,0) node{(b)};

	\end{tikzpicture}
	\caption{A graph $G$ with $\mm(G) = 3$, shown with (a) $V(M_1) = V(G) - \{v_3\}$ and (b) $V(M_2) = V(G) - \{v_1\}$.} \label{fig:examplemm}
\end{figure}
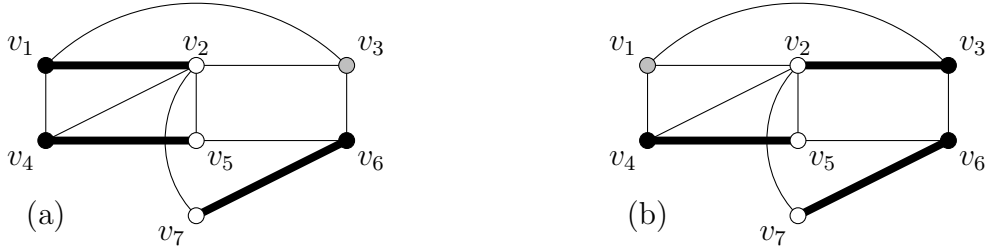

We return to the same graph $G$. In Figure~\ref{fig:examplemm}(a), we see one example of a matching with the uniqueness property, $M_1$, given in bold edges. The partition of the vertices of $M_1$ are colored black and white, respectively, and certainly $M_1$ is the unique matching given by these vertices. The vertex $v_3$ is not used by either part, so is colored gray. Note that  with that particular partition, $M_1$ is unique, but with another partition (as in Figure~\ref{fig:examplemm}(b)) another matching with the uniqueness property is possible, $M_2$. In either case, we see that $\mm(G) \ge 3$, and certainly no matching of size 4 is possible, so $\mm(G) = 3$.

Note that in the case of deduction, we can choose and fire multiple vertices simultaneously.  Thus the constrained zero forcing, deduction and constrained fast-mixed search procedures are all variants on the same idea, where the significant \emph{constraint} is the requirement that we can only choose vertices from the layout to fire, and a vertex (or searcher) can fire at most once. 
One of the main theorems of this paper (Theorem~\ref{thm_main}) states that for a connected graph $G$:
$$\zf(G) = \ms(G) = \ded(G) =  |V(G)| - \mm(G).$$

The remainder of the paper proceeds as follows.  In Section~\ref{Sec:ded}, we discuss properties of deduction, including variations in which we may choose which vertices fire in a given stage. We also explore the layout produced by the positions of searchers at the end of the deduction process. In Section~\ref{Sec:Equivalence}, we prove the equivalence theorem relating the parameters for constrained zero forcing, constrained fast-mixed search, deduction and matchings with the uniqueness property.  Section~\ref{Sec:Basic} considers the behavior of the parameter $\zf(G)$ under various operations, including taking subgraphs and adding or deleting edges, using these results to give a spectrum result for $\zf(G)$; 
the section also establishes the independence number of $G$ as a lower bound for $\zf(G)$. In Section~\ref{Sec:Dismantle}, we consider 
two graph families of interest -- one defined using the operation of 
successively deleting pendent edges, and the other defined by successively joining cliques.
  Section~\ref{Sec:Complexity} gives an NP-completeness result for $\zf(G)$, while Sections~\ref{Sec:Algorithm} and~\ref{Sec:Cacti} explore families of graphs for which $\zf(G)$ can be computed in polynomial time. In Section~\ref{Sec:Conclusion}, we conclude with some discussion and open problems. 
Note that we will write our results in terms of the constrained zero forcing number, but typically write proofs in one of the other models. In Sections 2--5, we will prove results primarily through the use of the deduction model (noting when we do not). However, in Sections 6--8, our proofs will be written in the context of fast-mixed searching, as that model more closely follows the existing algorithmic complexity literature.

\section{Standard layouts and free deduction} \label{Sec:ded}

In fast-mixed search, only one searcher may occupy a given vertex.  Though the deduction process allows multiple searchers per vertex, 
this is in fact unnecessary, as we will be able to achieve the minimum number of searchers by putting at most one searcher on a vertex;
we call a layout \dword{standard} if it has at most one searcher per vertex.
The following result, which was proved in~\cite{BDF}, shows that we may assume that layouts are standard.

\begin{thm} \label{thm:onecoppervertex} \cite{BDF} 
For the deduction game,
if there is a successful layout with $k$ searchers then there is a successful standard layout with $k$ searchers. 
\end{thm}

In the definition of the deduction game, at each stage we fire all the vertices that can fire.  A more flexible version allows any choice of fireable vertices to fire.  By the
\dword{free deduction} game we mean the usual deduction game, except that at each stage any set of fireable vertices may be chosen. From this point on, when we refer to deduction, we will be referring to the free deduction version unless otherwise stated.

\begin{figure}[htb]
	\centering
	\begin{tikzpicture}[x=1cm,y=1cm,scale=1]
	
	\newcommand*{\xshift}{8}%
	\newcommand*{\yshift}{4}%
	
	\foreach \x in {0,\xshift}
	\foreach \y in {0,0}
	{
		\draw (\x,\y+2) -- (\x + 2,\y+2) -- (\x +4,\y+2);
		\draw (\x +0,\y+1) -- (\x +2,\y+1) -- (\x +4,\y+1);
		\draw (\x +0,\y+1) -- (\x +0,\y+2);
		\draw (\x +2,\y+1) -- (\x +2,\y+2);
		\draw (\x +4,\y+1) -- (\x +4,\y+2);
		\draw (\x +0,\y+1) -- (\x +2,\y+2);
		\draw (\x +2,\y+0) -- (\x +4,\y+1);
		\draw (\x +2,\y+0) to [out=135, in=225] (\x +2,\y+2); 
		\draw (\x +0,\y+2) to [out=45, in=135] (\x +4,\y+2); 
		
		\draw[fill=black] (\x +0,\y+2) circle (3pt) node[above left]{$v_1$};
		\draw[fill=black] (\x +4,\y+2) circle (3pt) node[above right]{$v_3$};
		
		\draw[fill=black] (\x +0,\y+1) circle (3pt) node[below left]{$v_4$};
		\draw[fill=black] (\x +4,\y+1) circle (3pt) node[below right]{$v_6$};
		
		\draw[fill=black] (2+\x,2+\y) circle (3pt) node[above ]{$v_2$};
		\draw[fill=black] (2+\x,1+\y) circle (3pt) node[below right]{$v_5$};
		\draw[fill=black] (2+\x,0+\y) circle (3pt) node[below left]{$v_7$};

	}

	\draw (2,2) node[below right]{$s$};
	\draw (4,2) node[below right]{$s$};
	\draw (2,1) node[below left]{$s$};
	\draw (2,0) node[below right]{$s$};
	\draw (2+\xshift,2) node[below right]{$s$};
	\draw (0+\xshift,2) node[below right]{$s$};
	\draw (2+\xshift,1) node[below left]{$s$};
	\draw (2+\xshift,0) node[below right]{$s$};

	\draw (0,0) node{(a)};
	\draw (\xshift,0) node{(b)};

	\end{tikzpicture}
	\caption{Two more final configurations of deduction searchers in $G$ if  (a) the first vertex fired is $v_1$ and (b) the first vertex fired is $v_3$.} \label{fig:multided}
\end{figure}
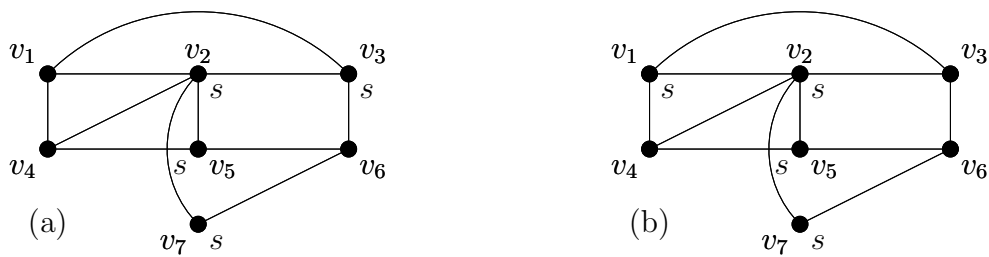

The next theorem shows that in the deduction model, the choice of what fires does not matter.
This basic result shows that deduction has a certain robustness, and
allows us to more easily prove facts about deduction.

\begin{thm} \label{thm_ded_all_same}
Given any layout, the vertices protected by any two free deduction games are the same.
\end{thm}

\begin{proof}

Let $S_1$ and $S_2$ be two distinct firing sequences originating from the same layout of vertices $L$.  
It suffices to show that every vertex protected by $S_1$ is protected by $S_2$.
Assume for contradiction that there is a vertex $x$ that is protected 
by $S_1$ but not by $S_2$.  
Suppose the firing sequence $S_1$ is $(A_1, A_2, \ldots, A_a)$ and $j$ ($1 \le j \le a$) is minimum such that $(A_1, \ldots, A_j)$ protects some vertex which $S_2$ never protects; let $x$ be such a vertex.
Let $A$ be the vertices protected by the initial firing of $S_1$ given by $(A_1, \ldots, A_{j-1})$ if $j > 1$, and if $j = 1$, we can just take $A$ to be the vertices in $L$ before any firing.  There is some $k$ so that the initial firing of $S_2$ given by $(B_1, \ldots, B_k)$ protects $B$, where
$A \subseteq B$.  Now consider vertex $x$; it is protected in $S_1$ because some vertex $y$ of $A_j$  fired a searcher into $x$.

We now arrive at the contradiction that $x$ would have to be protected by $S_2$.
In $S_2$, after $B_k$ has fired, since $A \subseteq B$, all the neighbors
of $y$ protected after the firing of $A_{j-1}$ in $S_1$, are also protected after the firing of $B_k$.  Since $y$ could fire in $S_1$, it can also fire in $S_2$, unless all the neighbors
of $y$ are fired into, in $S_2$, before $y$ gets a chance to fire.
Either way, $x$ is protected in $S_2$, providing our contradiction.
\end{proof}

Thus in proving facts about the deduction number we are free to choose what vertices fire at each stage.  A useful special case is to choose a single vertex to fire at each stage, calling this \dword{single-fire deduction}.

Consider again the initial layout given in Figure~\ref{fig:exampleded}(a). While in our previous example, we fired vertices $v_1$ and $v_3$ simultaneously, we now see that this is not necessary. If, instead, we fired only $v_1$, then the searcher on $v_3$ would be ``stranded'' for the remainder of the process. As indicated by Theorem~\ref{thm_ded_all_same}, however, when this process was completed, the entire graph would be protected, but the final configuration of searchers would be that shown in Figure~\ref{fig:multided}(a) instead of that of Figure~\ref{fig:exampleded}(d). On the other hand, if the first vertex fired is solely $v_3$, then the final configuration would be the one shown in Figure~\ref{fig:multided}(b). In short, even from a single successful layout, we may obtain many possible configurations of searchers once the process is complete.

We now note what happens to a successful layout as vertices are fired.
Note that it is false that firing only a single vertex of a successful layout leads to a successful layout.
For example, considering $C_4$ with two adjacent vertices occupied, firing a single vertex does not yield a successful layout.  However, if we continued until
all the vertices were protected, by firing the other vertex, we would have a successful layout.  So given a layout $L$, we say another layout $L^*$ is 
a \dword{terminal layout} of $L$, if $L^*$ is the result of some firing sequence which leads to all the vertices being protected. In the above example, we saw that the intermediate layouts between $L$ and $L^*$ may not be successful, but the
next theorem points out that $L^*$ is in fact successful, answering an open question from \cite{BDF}.

\begin{thm} 
Given any successful layout, any corresponding terminal layout is also successful.
\end{thm}

\begin{proof}

Suppose our graph is $G$ and we have layout $L$.
To get from $L$ to some terminal layout $L^*$ there is some firing sequence 
$(S_1, \ldots, S_k)$.  We argue that we can simply ``reverse'' this firing sequence to go from $L^*$ to $L$, where by reversing it we mean:
take the firing sequence $(S'_k, \ldots, S'_1)$, where $S'_j$ fires from the targets of $S_j$ back to $S_j$.
Now we argue this works by induction on $k$, the length of the firing sequence.

Consider the last firing, $S_k$, and suppose it fires into the unoccupied (and unprotected) vertices $S_k'$.
Then we can in fact fire $S_k'$ back to $S_k$, as follows.  Adjacent to 
$S_k'$ are the vertices of $S_k$ and some other vertices of $G$.  Among these other vertices, there cannot be an unoccupied protected vertex --- since such a vertex is protected because something fired from it earlier in the firing process, and so would have had to fire into the vertex of $S_k'$ in order to be allowed to fire.  This means that all the edges of concern are between $S_k'$ and $S_k$.

Consider the firing from $S_k$ to $S_k'$. For any vertex in $S_k'$, if it has $d$ neighbors in $S_k$, it received $d$ searchers in the original firing. Thus, it can fire back into $S_k$.

Now remove the vertices $S_k'$ from the graph, that is the unoccupied, but just protected vertices, to arrive at the graph $G^-$. We likewise define a layout $L^-$ based on $L$ but with these vertices removed. Thus $(S_1, \ldots, S_{k-1})$
is a successful firing sequence starting with $L^-$ in $G^-$.  By the inductive hypothesis we can reverse it to get that $(S'_{k-1}, \ldots, S'_1)$ is successful, and then
we can append $S'_k$ to the left to get our conclusion in $G$.
\end{proof}

\section{Equivalence Theorem}\label{Sec:Equivalence}

As noted in the introduction, there are many similarities between the constrained zero forcing, fast-mixed search and deduction processes.  The primary goal of this section is to establish an equivalence between the number of searchers/colored vertices required in these processes, as well as a structural characterization involving matchings with the uniqueness property.  In particular, the main result of this section is Theorem~\ref{thm_main}.

\begin{lemma} \label{lem:bipartite-graph}
If $G$ is a graph that contains a perfect matching with the uniqueness property and  $B$ is the bipartite graph 
 that witnesses this, 
then $B$ must contain a degree 1 vertex.
\end{lemma}

\begin{proof}
Assume for contradiction that $B$ is given by bipartition $(V_1, V_2)$ and every vertex in $B$ has degree at least 2.  Let $M = \{u_1 v_1, \ldots, u_k v_k  \}$ be a perfect matching in $B$. For a contradiction we find another perfect matching in $B$.  In fact, if we ever have an even cycle with alternating edges from $M$ (called this a \emph{vicious cycle}), then we can just use the edges in the cycle, but not in $M$, to arrive at our contradictory second perfect matching.
Since each $u_i$ has degree 2, we can proceed from $i= 1$ to $k$, considering neighbors of $u_i$ other than $v_i$.  We can assume $u_i$ has neighbor $v_{i+1}$ until some point (possibly when $i = k$) when some $u_j$ has only neighbors from among $\{ v_1, \ldots, v_j  \}$, at which point we have a vicious cycle, and hence a contradiction.
\end{proof}

\begin{thm} \label{thm_main}
If $G$ is a connected graph, then $$\zf(G) = \ms(G) = \ded(G) =  |V(G)| - \mm(G).$$
\end{thm}

\begin{proof} We show that (1) $\zf(G)=\ded(G)$, (2) $\ms(G)=\ded(G)$ and (3) $\ded(G)=|V(G)|-\mm(G)$. 
\medskip

\noindent
(1) Proof that $\zf(G) = \ded(G)$:

This is immediate from Theorems~\ref{thm:onecoppervertex} and~\ref{thm_ded_all_same} by noting that single-fire deduction starting with a standard layout 
is just constrained zero forcing.
\medskip

\noindent
(2) Proof that $\ms(G) = \ded(G)$: 

If we restrict to standard layouts for deduction, the sliding action in fast-mixed search is the same as the single-firing action in deduction.  
The major difference between the deduction and constrained fast-mixed search processes is that in deduction we are concerned with clearing vertices, while in fast-mixed search it is edges.  

It is easy to see that for any connected graph $G$, if all edges become cleared in constrained fast-mixed search, then the corresponding moves in deduction clear all vertices, since any cleared edge must either have both endvertices occupied at some point, or else a searcher moving from one endvertex to the other.  We now show that if all vertices are cleared in deduction (beginning with a standard layout), then the corresponding sliding actions clear all edges in fast-mixed search.  
Consider an edge $uv$ in $G$.  If $u$ and $v$ are both occupied in the initial layout or if $u$ and $v$ are both occupied at the end of the process, then the edge $uv$ becomes cleared.  Suppose now that one of these vertices, say $u$, is occupied in the initial layout and that a searcher moves to $v$ at some point in the process.  If the searcher on $u$ moves to $v$, then the edge $uv$ is cleared.  Otherwise, as $v$ is an unoccupied vertex adjacent to $u$, the searcher on $u$ cannot move to another vertex before $v$ becomes occupied; thus, at the point when a searcher moves to $v$, vertices $u$ and $v$ are both occupied so that the edge $uv$ is cleared.
\medskip

\noindent
(3) Proof that $\ded(G) = |V(G)| - \mm(G)$:

We show $\ded(G) =  |V(G)| - \mm(G)$, first showing $\ded(G) \leq |V(G)|-\mm(G)$.
Let $B=(V_1,V_2,E_B)$ be a bipartite subgraph of $G$ with a perfect matching of size $\mm(G)$ which witnesses the uniqueness property. 
We construct a successful layout $L$ for deduction containing $|V(G)| - |E_B|$ 
searchers as follows.  Place a searcher on each vertex of $V(G)-V_2$; that is, a searcher occupies each vertex of $V_1$ as well as each vertex not in $B$.  We show that this layout is successful by induction on $|V_1|$.  By Lemma~\ref{lem:bipartite-graph}, we may assume without loss of generality that $V_1$ contains a vertex $u$ of degree 1 in $B$. The neighbor $v$ of $u$ in $B$ is the only unprotected neighbor of $u$ in $G$, and so in the first stage the searcher on $u$ moves to $v$.  By Theorem~\ref{thm_ded_all_same}, we may consider a free deduction game in which only $u$ moves in the first stage.  Let $G'=G-\{u,v\}$, and $B'=B-\{u,v\}$.  Then $B'$ contains a 
unique perfect matching, so by the induction hypothesis, the layout $L$ restricted to $G'$ is successful in $G'$.  It follows that $L$ is a successful layout in $G$, and thus $\ded(G) \leq |V(G)|-\mm(G)$.

We now show that $|V(G)|-\mm(G) \leq \ded(G)$.  Consider a standard successful layout $L$ in $G$.  By Theorem~\ref{thm_ded_all_same}, $L$ is successful in 
single-fire deduction, so we henceforth assume this is the case.  In particular, no two searchers will move to protect the same vertex, so that the edges on which searchers move form a matching $M$.  Let $S$ be the set of vertices whose searchers move at any point, let $T$ be the set of their target vertices, and let $B$ be the bipartite graph induced by the edges with one endvertex in $S$ and the other in $T$.  

We show that $M$ has the uniqueness property.  
Let $(s_1, s_2, \ldots, s_{\ell})$ be the sequence of vertices fired in $S$, and for each $i \in \{1, \ldots, \ell\}$, let $t_i \in T$ be the vertex to which the searcher on $s_i$ moves; that is, $M$ consists of the edges $\{s_i,t_i\}$, $i \in \{1, \ldots, \ell\}$.  Since $s_1$ is the first vertex fired, its only neighbor in $B$ can be $t_1$; thus $\{s_1,t_1\}$ must be in any perfect matching in $B$.  Since $s_2$ fires immediately after $s_1$, it must be that $t_2$ is the only vertex of $T\setminus\{t_1\}$ adjacent to $s_2$; thus $\{s_2,t_2\}$ is in any perfect matching in $B$. Continuing inductively, we see that $M$ is the unique perfect matching in $B$. 
\end{proof}

In the remainder of this paper, we will state our results in terms of the constrained zero forcing number.  However, as a consequence of Theorem~\ref{thm_main}, to prove these results we may use any of the constrained zero forcing, deduction or constrained fast-mixed search processes, or the structural characterization regarding matchings with the uniqueness property.

\section{Basic Properties} \label{Sec:Basic}

In this section we prove some basic properties of the parameter $\zf(G)$.  We will consider how it behaves with respect to subgraphs and operations on graphs, prove additional bounds on the parameter and give a spectrum result.  In many of the proofs in this section, we approach results on $\zf(G)$ using the deduction process.

The parameter $\zf(G)$ is not monotonic with respect to taking subgraphs in general; that is, if $H$ is a subgraph of $G$, then $\zf(H)$ may be greater or less than $\zf(G)$.  To see this, note that $K_{1,n-1}$ is a subgraph of the wheel $W_n$ of order $n$, which is in turn a subgraph of $K_n$.  In~\cite{BDF}, it is noted that $\zf(K_{1,n-1}) = n-1 = \zf(K_n)$, while $\zf(W_n) = \lceil \frac{n}{2} \rceil$.  In contrast, the following result shows that induced subgraphs behave much more predictably.

\begin{thm} \label{thm:subgraph}
If $H$ is an induced subgraph of $G$ then $\zf(H) \le \zf(G)$.
\end{thm}

\begin{proof}
We consider this as a deduction problem.
Let $L$ be a successful layout on $G$ with $\ded(G)$ searchers. By Theorem~\ref{thm:onecoppervertex}, we may assume that $L$ is a standard layout.  
We construct a successful layout $L'$ in $H$ with at most $\ded(G)$ searchers as follows.  Each vertex of $H$ which is occupied in $L$ remains occupied in $L'$.  Additionally, place a searcher on each vertex of $H$ which is a target of a vertex from $L$ in $V(G) \setminus V(H)$.  

We show that $L'$ is successful.  Consider any
successful single-firing sequence for $L$ in $G$; call it $A = (a_1, \ldots, a_k)$.
Let $B = (b_1, \ldots, b_j)$ be the result of removing any $a_i$ from $A$ 
if $a_i$ is in $V(G) \setminus V(H)$, but otherwise maintaining the order.
We show that $B$ is a successful single-firing sequence for $L'$ in $H$.
Whenever there is call to fire some $b_i$ in $B$, it can fire since the same vertex $b_i$ appears in $A$, where it could fire; since $H$ is induced, it still has any necessary edge to move along (and note that if in $A$, $b_i$ fired into $V(G) \setminus V(H)$, then $b_i$ can stay still in $B$).  In particular, the appropriate neighbors of $b_i$ in $B$ are protected (i.e. all but one neighbor) either by previous firings or by the placement of some fired searcher from $A$ directly at its target vertex. 
Since the firing of $A$ protected all of $G$, the firing of $B$ protects all of $H$.
\end{proof}

The following basic bounds on $\zf(G)$ appear in~\cite{BDF}.
\begin{thm}\cite{BDF}
For any graph $G$ of order $n$ with at least one edge, $\left\lceil \frac{n}{2} \right\rceil \leq \zf(G) \leq n-1$.
\end{thm}
These bounds are both tight; as pointed out in~\cite{BDF}, paths achieve the lower bound, while complete graphs achieve the upper bound.  A natural spectrum question arises:  given an integer $n \geq 2$ and an integer $d$ such that 
$\left\lceil \frac{n}{2} \right\rceil \le d \le n-1$, is there a graph of order $n$ with $\zf(G)=d$?  To answer this question, we first consider the effect of adding or deleting an edge on the value of $\zf(G)$.  We then answer the question affirmatively in Corollary~\ref{corSpectrum}.

\begin{thm} \label{thm_addedge}
If $G$ is a graph and $G^*$ is $G$ with an edge added or removed between two vertices of $G$, then $\zf(G) - \zf(G^*)$ can only be $0$, $+1$, or $-1$. 
\end{thm}

\begin{proof}
This proof uses the deduction characterization.
Let  $x$ and $y$ be distinct vertices in $G$, and let $G^*$ be either of the following: the result of adding the edge $e = \{x, y\}$ if $G$ did not contain $e$, or the result of removing $e$, if $G$ contained $e$.  Our goal is to show the following: 
$$\ded(G) - 1 \le \ded(G^*) \le \ded(G) + 1.$$

Let $L$ be a standard successful layout on $G$ using \ded(G) searchers.
We first note the following (call it the \emph{first inequality}): $$\ded(G^*) \le \ded(G) + 1.$$
To see this, consider the case in which  
no searcher fires from $x$ or $y$ in $L$. 
Then the presence or absence of edge $e$ is irrelevant to the success of $L$, so the same layout works in $G^*$.  In the case that a searcher fires from at least one of $x$ or $y$ in $L$, we can simply add a searcher to the other vertex of $e$ in $G^*$ to ensure that searcher can still fire,  or that both $x$ and $y$ are protected in the case that a searcher fires from one to the other in $G$.

Now we show the following: $$\ded(G) - 1 \le \ded(G^*).$$
To see this, 
we can reverse the role of $G$ and $G^*$, adding or removing edge $e$ to/from $G^*$ in order to arrive back at $G$.  Applying the first inequality above,
we have $\ded(G) \le \ded(G^*) + 1$, as required. 
\end{proof}

Note that Theorem~\ref{thm_addedge} answers a question posed in~\cite{thesis}, which asked how large the difference between $\ded(G^*)$ and $\ded(G)$ can be.  
We can apply Theorem~\ref{thm_addedge} to obtain the next corollary, which
shows that for graphs $G$ on $n$ vertices,
$\zf(G)$ can take on all the values between the minimum possibility of 
$\lceil \frac{n}{2} \rceil$ and the maximum possibility of $n-1$.
This follows by starting with $C_n$ and
successively adding one edge at a time to get $K_n$; note that $\zf(C_n) = \lceil \frac{n}{2} \rceil$ and $\zf(K_n)=n-1$~\cite{BDF}.  
Since Theorem~\ref{thm_addedge} shows that each edge can increase $\zf$ at most 1, we must achieve all the intermediate values.

\begin{cor} \label{corSpectrum}
For any order $n$ and integer $d$ with $\lceil \frac{n}{2} \rceil \leq d \leq n-1$, there is a graph $G$ with $\zf(G)=d$.
\end{cor}

Given a graph of order $n$, we may ask what the maximum number of edges that can be consecutively added which each increase, decrease or do not change the value of the constrained zero forcing parameter.  Since this parameter is bounded between $\lceil \frac{n}{2} \rceil$ and $n-1$, clearly the number of edges that can be added so that each decreases the deduction number, is bounded above by $\lfloor \frac{n}{2} \rfloor$.  Indeed, this bound can be attained.  To see this, consider the star $S_{n}$ with $(n-1)$ leaves $\ell_1, \ldots, \ell_{n-1}$, where $n$ is even.  Adding in turn the edges $\ell_1 \ell_2, \ell_3 \ell_4, \ldots, \ell_{n-3}\ell_{n-2}$ decreases the deduction number by 1 each time.

An example of adding linearly many edges in the number of vertices, each of which increases the deduction number by $1$, is as follows.  Start with a star $S_{m+2}$, centered at vertex $v$ and with $m+1$ leaves.  To all but one leaf, attach a copy of $K^-_4$, where $K^-_4$ is $K_4$ with an edge removed, but no vertices removed; see Figure~\ref{K4-ex}.  

\begin{figure}[htb]
\centering
\begin{tikzpicture}[x=1cm,y=1cm,scale=1]
\draw[fill=black] (3,4) circle (3pt);
\draw[fill=black] (6,4) circle (3pt) node[above right]{$v$};
\draw (3,4) -- (6,4); \draw (6,4) -- (1,2); \draw (6,4) -- (6,2);  \draw (6,4) -- (12,2);
\draw (0,1) -- (2,1); \draw (0,1) -- (1,0);  \draw (0,1) -- (1,2); \draw (2,1) -- (1,0); \draw (2,1) -- (1,2); 
\draw (5,1) -- (7,1); \draw (5,1) -- (6,0);  \draw (5,1) -- (6,2); \draw (7,1) -- (6,0); \draw (7,1) -- (6,2); 
\draw (11,1) -- (13,1); \draw (11,1) -- (12,0);  \draw (11,1) -- (12,2); \draw (13,1) -- (12,0); \draw (13,1) -- (12,2); 
\foreach \x in {0,2,5,7,11,13}{
	\draw(\x,1)[fill=black]circle (3pt);
}
\draw (9,1) node{$\ldots$};
\foreach \x in {1,6,12}{
	\draw[fill=black] (\x,0) circle (3pt);
	\draw[fill=black] (\x,2) circle (3pt);
}
\end{tikzpicture}
\caption{A graph $G$ to which a linear number of edges can be added which increase the deduction number.} \label{K4-ex}
\end{figure}
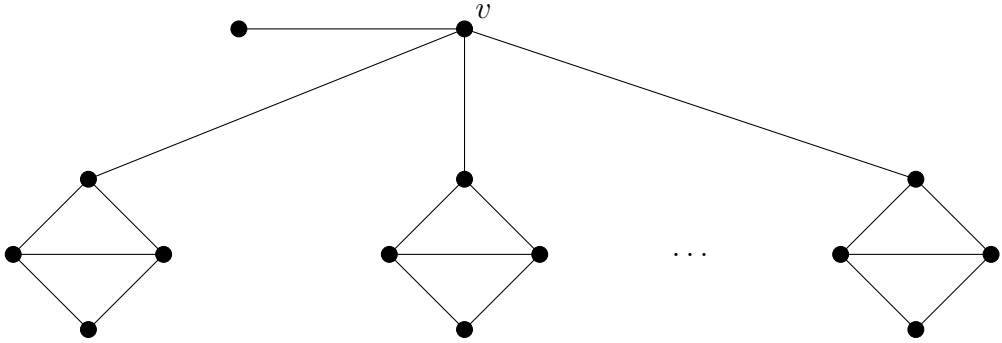
The graph $G$ obtained has $4m+2$ vertices and $\zf(G)=2m+1$: one searcher is needed to clear the remaining leaf, and this searcher also clears $v$; each copy of $K^-_4$ requires two additional searchers.  If we add the ``missing'' edge from a copy of $K_4$, the deduction number increases by 1, as three searchers become necessary to clear the resulting copy of $K_4$.  Thus, we may add $m = \lfloor \frac{|V(G)|}{4}\rfloor$ edges sequentially which each increase the deduction number.

On the other hand, the number of edges we can add which do not affect the deduction number may be quadratic in the number of vertices.  To see this, consider the following graph $G$ with vertex set $V(G)=A \cup B$ where $\{a_1, \ldots, a_m\}$ and $B=\{b_1, \ldots, b_m\}$, and edge set $E(G) = \{a_1 x \mid x \in V(G) \setminus \{a_1\}\} \cup \{a_ib_i \mid 2 \leq i \leq m\}$.  (Note that $G$ may be viewed as constructed from a star centred at $a_1$ by adding a matching joining pairs of leaves.)  It is straightforward to verify that $\zf(G)=m = \frac{|V(G)|}{2}$; one successful layout is to place a searcher on each vertex of $A$.  Adding edges sequentially of the form $a_i a_j$ or $b_ib_j$ does not change the deduction number, nor does subsequently adding edges of the form $a_i b_j$ where $i \leq j$.  Each of these additions still produces a graph with a matching of size $m$ with the uniqueness property witnessed by the bipartite subgraph induced by the edges between $A$ and $B$.  The total number of edges we have added to $G$ is
\[
\binom{m-1}{2} + \binom{m}{2} + \binom{m-1}{2} = \frac{(m-1)(3m-4)}{2}. 
\]

We conclude this section by discussing a lower bound on $\zf(G)$ supplemental to the results of~\cite{BDF}.  
Given a graph $G$, a subset of ${\cal I} \subseteq V(G)$ is an \dword{independent set} of $G$ if there is no edge in $G$ that connects two vertices of ${\cal I} $.
A \emph{maximum independent set} of $G$ is an independent set of $G$ with the largest possible size; this size is called the \dword{independence number} of $G$, denoted by $\alpha(G)$.
The following theorem establishes a relation between $\zf(G)$ and $\alpha(G)$.

\begin{thm}  \label{thm:LB-max-indep}
If $G$ is a graph, then $\zf(G) \geq \alpha(G)$.
\end{thm} 

\begin{proof} We again consider the deduction model.
	For a graph $G=(V,E)$, let ${\cal I}$ be a maximum independent set.  Since there is no edge in $G$ connecting two vertices of ${\cal I}$, every searcher must move between a vertex of ${\cal I}$ and a vertex of $V \setminus {\cal I}$ or between two vertices of $V \setminus {\cal I}$. 
	Consider a standard successful layout of $G$ using the minimum number of searchers. Let $M_1$ be a set of edges between vertices of ${\cal I}$ and vertices of $V \setminus {\cal I}$ such that a searcher traverses these edges, and let $M_2$ be a set of edges between vertices of $V \setminus {\cal I}$ that a searcher traverses. Any vertex that is not a target vertex or a fired vertex contains an unmoving searcher, of which there are $|V| - |M_1| - |M_2|$. 
	Note that $|V| \geq |{\cal I}| + |M_1| + 2|M_2| \geq |{\cal I}| + |M_1| + |M_2|$. Hence $\ded(G) = |V| - |M_1| - |M_2| \geq  |{\cal I}|.$
\end{proof}

The difference $\zf(G)-\alpha(G)$ can be arbitrarily large, as can the ratio $\frac{\zf(G)}{\alpha(G)}$; this can be observed by considering the complete graph $K_n$.  However, we note that the bound of Theorem~\ref{thm:LB-max-indep} is tight.  In~\cite{BDF}, the deduction number of trees is characterized by producing a standard layout by successively placing cops on leaves and ``pruning'' the tree.  It is easily verified that in this layout, the set of occupied vertices is independent.  We thus have the following corollary.

\begin{cor} \label{cor:treeindep}
    If $T$ is a tree, then $\zf(T)=\alpha(T).$
\end{cor}

\section{Dismantlable graphs} \label{Sec:Dismantle}

Considering trees, and particularly the result of Corollary~\ref{cor:treeindep} and the pruning method of \cite{BDF}, we see that pendent edges often play an important role in these problems. In deduction, each pendent edge must contain a searcher, and that searcher effectively removes the edge, as we will see in Theorem~\ref{thm:G-leaf}. Of course, this removal may produce new pendent edges, and so on. Likewise, a similar argument may be made with cliques. We examine both techniques.

\subsection{Pendent edge-dismantlable}

By a \dword{pendent edge} in a graph $G$ we mean an edge $e = uv$ such that at least one of its endpoints has degree 1. To delete an edge $e$, we write $G - e$ to mean $G$ with edge $e$ removed, along with both vertices $u$ and $v$.
We develop a procedural way to view \zf \  on certain graphs --- 
pendent edges are successively removed until no edges remain.  The process in fact describes a way to find a matching with uniqueness property and thus to compute \zf.
We first connect the structural property of a pendent edge to the \zf \ parameter.

\begin{thm} \label{thm:G-leaf}
If $G$ is a graph that contains a pendent edge $e$, then 
$$\zf(G) = \zf(G-e)+1.$$

\end{thm} 

\begin{proof}
We consider the constrained fast mixed search model.
Let $e = vu$ where $v$ has degree 1, and hence $N[v] = \{u,v\}$. Let $G' = G - e$. If $\degree_{G}(u)=1$, the claim is trivial.  Suppose $\degree_{G}(u) > 1$. Let $S$ (resp. $S'$) be an optimal search strategy for $G$ (resp. $G'$)  in the constrained fast-mixed search model. There are two possible cases for clearing $uv$ by $S$. 
\medskip

{\sc Case 1.} In $S$, $uv$ is cleared by sliding a searcher from $u$ to $v$ or from $v$ to $u$. 

In this case, each edge incident on $u$ except $uv$ is cleared by two searchers occupying its two endpoints. If we delete the action of placing a searcher on $u$ or $v$ and delete the action of sliding a searcher from $u$ to $v$ or from $v$ to $u$, then the remaining actions of $S$ can clear $G'$. Thus $\ms(G') \leq \ms(G) - 1$.
On the other hand, note that we can also clear $G$ in the following way: place a searcher on $u$ at the beginning, then clear $G'$ by $S'$, and finally slide the searcher on $u$ to $v$. Hence $\ms(G) \leq \ms(G')+1$, and therefore $\ms(G) = \ms(G')+1$.
\medskip

{\sc Case 2.} 
In $S$, $uv$ is cleared by two searchers on $u$ and $v$. 

If the searcher on $u$ is placed on this vertex, then we do not need to place another searcher on $v$ because the searcher on $u$ can slide to $v$ in the end. This contradicts the assumption that $S$ is an optimal search strategy for $G$. So in Case 2, a searcher slides from a neighbor of $u$, say $u'$, to $u$. Note that $S$ contains two placing actions - ``placing a searcher on $u'$'' and ``placing a searcher on $v$''. In $S$, if we replace the action ``sliding a searcher from $u'$ to $u$'' by the new action ``sliding a searcher from $v$ to $u$'', then the new search strategy is also an optimal search strategy for $G$. Since in the new optimal search strategy of $G$, $uv$ is cleared by sliding a searcher from $v$ to $u$, it follows from Case 1 that $\ms(G) = \ms(G')+1$.
\end{proof}

Note that corresponding to Theorem~\ref{thm:G-leaf}, only $\ded(G) \leq \ded(G')+1$ is proved in  \cite{thesis} (see Corollary 6.1.3 and Theorem 6.1.4 when $G'=K_n$).

\begin{cor} \label{cor:G-leaves}
Let $G$ be a graph containing at least one pendent edge. If $H$ is a graph obtained from $G$ by successively deleting $k$ pendent edges, then $\zf(G) = \zf(H) + k$.
\end{cor}

The following theorem describes, as an application of pendent-edge removal, a method of obtaining graphs for which $\zf(G)$ meets the bound of Theorem~\ref{thm:LB-max-indep}.

\begin{thm}  \label{thm:subdivision-1}
Let $G$ be a connected graph and $G'$ be a subdivision of $G$ obtained by subdividing every edge exactly once. 
\begin{enumerate}
\item[(i)]
If $G'$ is a tree, then $\zf(G') = |E(G)| +1 = \alpha(G')$. 
\item[(ii)]
If $G'$ contains a cycle, then $\zf(G') = |E(G)| = \alpha(G')$. 
\end{enumerate}
\end{thm} 

\begin{proof}  We use the constrained fast-mixed search model. 
If  $G$ contains only an isolated vertex, then $\ms(G') = \ms(G) = 1$, and so the claim is true. 
Suppose $G$ contains at least one edge. In $G'$, all vertices of $G$ are called 
\dword{base vertices} and all other vertices are called \dword{middle vertices}.  A vertex of degree one is called a leaf.
If $G'$ contains a leaf $v$, let $v'$ be the neighbor of $v$ in $G'$ which must be a middle vertex. 
In $G'$, place a searcher on $v$ and slide this searcher from $v$ to $v'$. Then we remove 
$v$ and $v'$ from $G'$. Note that the removal of these two vertices can make at most one non-leaf base vertex of $G'$ become a leaf base vertex of the new graph. 
We can keep doing these actions until the remaining graph, denoted by $H$, has no leaf. 
Let $k$ be the number of sliding actions in the above process. Thus, by Corollary~\ref{cor:G-leaves}, $\ms(G') = \ms(H) + k$.
Since $G$ is connected, $H$ must also be connected. 

If $H$ contains only an isolated vertex $u$, then $G$ must be a tree.
We can simply place a searcher on $u$ to clear $H$.  
Thus $\ms(G') = k +1$. Notice that when we reduce $G'$ to $H$, we delete $k$ edges and their endpoints from $G'$, where each of these $k$ edges corresponds to an edge of $G$. 
Thus, if $G$ is a tree, then $|E(G)| = k$, and therefore $\ms(G') = |E(G)| +1$.

Suppose $H$ contains at least one edge. Note that the minimum degree of $H$ is $2$. We have two cases regarding the maximum degree of $H$.
\medskip

{\sc Case 1.} If the maximum degree of $H$ is 2, then $H$ is a cycle of even length.   It follows from~\cite[Theorem~4.2]{BDF} that $\ms(H) = |V(H)|/2$. Let $C$ be the cycle in $G$ such that $H$ is the subdivision of $C$. Thus $\ms(H) = |E(C)|$. For the $k$ edges deleted from $G'$ when we reduce  $G'$ to $H$, they correspond to the $k$ edges in $G - E(C)$.
Hence $\ms(G') = \ms(H) + k = |E(G)|$.
\medskip

{\sc Case 2.} If the maximum degree of $H$ is greater than $2$, let $u$ be a vertex of $H$ with the maximum degree.  
Let $v$ be a base vertex in $H$ such that the distance between $u$ and $v$ is two. Let $v'$ be the middle vertex in $H$ that is adjacent to both $u$ and $v$. 
In $H$, place a searcher on $v$ and place a searcher on each neighbor of $v$ except $v'$, and then slide the searcher on $v$ to $v'$. Then we delete 
$v$ and $v'$ from $H$. Note that the deletion of these two vertices cannot make the vertex $u$ become a leaf in the remaining graph. Slide the searchers on the neighbors of $v$ except $v'$ to their neighboring base vertices, and delete these vertices and their neighboring base vertices. Note that the deletion of these vertices cannot make any base vertex of $H$ become a leaf in the remaining graph.  Let $H'$ denote the remaining graph. Let $X$ be an arbitrary component in $H'$. If $X$ is an isolated vertex, then place a searcher on it; otherwise,
$X$ must contain a leaf $x$ that is a middle vertex. Place a searcher on $x$ and slide this searcher to its neighboring base vertex $x'$. Remove $x$ and $x'$ from $X$.
We can keep doing in this way until either all vertices of $X$ are deleted or each new component contains only an isolated vertex, which must be a middle vertex. We place a searcher on each of these isolated vertices to clear $X$. The total number of searchers placed on $H$ is equal to the number of middle vertices in $H$. 
When we reduce $G'$ to $H$, we have deleted $k$ edges and their endpoints from $G'$. 
These $k$ edges correspond to $k$ middle vertices in $G'$ that are not in $H$.  
Note that each middle vertex of $G'$ corresponds to an edge of $G$. 
Hence $\ms(G') = \ms(H) + k \leq |E(G)|$.
Notice that the set of middle vertices in $G'$ is an independent set. From Theorem~\ref{thm:LB-max-indep}, we have  $\ms(G') \geq |E(G)| $. 
Therefore $\ms(G') = |E(G)|$ if $G$ is not a tree.

As for the independence numbers, it is immediate that if $G$ is a tree then 
$\alpha(G') = |E(G)| + 1$, by taking the base vertices as our independent set.
If $G$ has a cycle, then we already have $\alpha(G') \le \zf(G') = |E(G)|$, so
we just need to note that $|E(G)| \le \alpha(G')$, which is immediate by taking the 
middle vertices as our independent set.
\end{proof}

We now define a class of graphs $G$ for which $\zf(G)$ is characterized by the operation of pendent-edge removal.  As a point of notation, a graph $N=(V,E)$ is called a \emph{null graph} if $V=\emptyset$ or  $E=\emptyset$.

\begin{defn} \label{def:dismantlable}
Suppose $G$ is a graph and $(e_1, \dots, e_k)$ is a sequence from $E(G)$; let $G_0 = G$ and for 
for $1 \le i \le k$, let $G_i = G - \{ e_1, \dots, e_i \}$.
\begin{itemize}

\item
We say that $(e_1, \dots, e_k)$ is a \dword{pendent-edge dismantling ordering} (of length $k$) if for $1 \le i \le k$, $e_i$ is a pendent edge in $G_{i-1}$, and $G_k$ is
 a null graph.
 
 \item
A graph is \dword{pendent-edge dismantlable} if it has a 
pendent-edge dismantling ordering.
\end{itemize}
\end{defn}

For example, any forest is pendent-edge dismantlable.  Also,
a $C_5$ with a single edge attached to one of its vertices is also pendent-edge dismantlable.  But note that a $C_5$ with a length 2 path attached to one of its vertices is \emph{not} pendent-edge dismantlable, since deleting the only pendent edge leaves just the $C_5$, which has no pendent edges.

We prove the following fact which generalizes the pruning algorithm for trees given in~\cite{BDF}.

\begin{thm} \label{thm:A-dismantlable}
If $G$ is an $n$-vertex graph with a length $k$ pendent-edge-dismantling ordering,  
then  $\zf(G) = n - k$.
\end{thm} 

\begin{proof}
Suppose the graphs corresponding to the pendent-edge-dismantling ordering
are $G_0$, $G_1$, $\ldots$, $G_k$, as in Definition~\ref{def:dismantlable}.
By repeated application of Theorem~\ref{thm:G-leaf},
$$\zf(G) = \zf(G_0) = \zf(G_1) + 1 = \zf(G_2) + 2 = \cdots = \zf(G_k) + k = n - 2k + k = n - k,$$
where we note that $\zf(G_k) = n - 2k$ for the following reason:
 two vertices are lost at each step from $G_i$ to $G_{i+1}$, so 
$G_k$ has $n - 2k$ vertices, and since 
$G_k$ is a null graph all $n - 2k$ of its vertices are included.
\end{proof}

Thus we can see that a pendent-edge dismantling ordering as a particular way 
to find a maximum matching with the uniqueness property. Thus, if a graph $G$ has 
a dismantling ordering of length $k$, then $k = \mm(G)$.
So for graphs with a pendent-edge dismantling ordering, we can consider any dismantling ordering (all of the same length, say $k$)  and then $\zf(G) = n - k$.

\subsection{Clique-constructable}

We have described pendent-edge dismantlable graphs as those that can be reduced to a null graph by successive deletion of pendent edges.  We may also view these graphs as being built recursively: beginning with a pendent-edge dismantlable graph $G$, we can construct a new pendent-edge dismantlable graph by adding a new vertex $u$ adjacent to some subset of the vertices of $G$, as well as a vertex $v$ adjacent only to $u$.  It is in this vein of recursive construction that we describe a second class of graphs based on cliques.

\begin{defn} \label{def:k-clique-dismantlable}
Let $G$ be a graph, and let $(Z_1, Z_2, \ldots, Z_k)$ be a sequence of pairwise vertex-disjoint cliques in $G$, each of order at least $2$.  For each $1 \leq i \leq k$, let $u_i$ be a vertex in $Z_i$.  

Let $G_1=Z_1$. For each $1 \leq i \leq k-1$, form $G_{i+1}$ by choosing a clique $Y_{i}$ in $V(G_{i}) \setminus \{u_1, \ldots, u_{i}\}$, and adding $Z_{i+1}$ along with all edges of the form $xy$ where $x \in V(Z_{i+1})$ and $y \in V(Y_{i})$.  

We say that $(Z_1, Z_2, \ldots, Z_k)$ is a {\bf clique-construction ordering (of length $k$)} of $G$ if for some choice of vertices $u_1, \ldots, u_k$ and cliques $Y_1, \ldots, Y_{k-1}$ in this process, $G=G_k$. A graph $G$ which has a clique-construction ordering is called {\bf clique-constructable}.
\end{defn}

Note that by successively deleting the cliques $Z_k, Z_{k-1}, \ldots, Z_{2}$ from $G$, we obtain the sequence of clique-constructable graphs $G_{k-1}, G_{k-2}, \ldots, G_1$, where $G_1=Z_1$ is itself a clique.

\begin{thm} \label{thm:edge-dismantlable-chordal}
If $G$ has a clique-construction ordering of length $k$ then 
$\zf(G) = |V(G)| - k$.
\end{thm} 

\begin{proof}
Let $(Z_1, \ldots, Z_k)$ be a clique-construction ordering for $G$.
If $k=1$, then $G$ is a complete graph containing at least one edge, and thus $\zf(G) = |V(G)| - 1$, as required. Suppose $k \geq 2$.  For each $i \in \{1, \ldots, k\}$, let $v_i$ be some vertex in $V(Z_i) \setminus \{u_i\}$.  
From Definition~\ref{def:k-clique-dismantlable}, let $U_1=G_1$, and every time when we add a clique $Z_{i+1}$, we obtain a clique $U_{i+1}$ induced by the vertices of $Y_{i}$ and $Z_{i+1}$.  
	Note that each $U_i$ is a clique and every edge of $G$ is in at least one $U_i$. 

We first use the deduction model to show that $\ded(G) \leq |V(G)|-k$,   Place a searcher on each vertex of $G$ except $\{u_1, \ldots, u_k\}$.  Note that all vertices in $N[v_{k}]$ except $u_k$ are occupied by searchers, so the searcher on $v_k$ can move to $u_k$.  Thus every vertex in $U_k$ is protected.  Similarly, we can protect every vertex of $U_{i}$, $i=k-1, \dots, 1$, by moving the searcher on $v_i$ to $u_i$. Thus 
$\ded(G) \leq |V(G)| - k$. 

To prove the equality, we consider the constrained fast mixed search model. 
For the sake of contradiction, suppose that $\ms(G) < |V(G)| - k$. 
Consider an optimal strategy ${\cal S}$ in which all placing actions are before all sliding actions. Suppose that ${\cal S}$ contains $\ell$ sliding actions, and let $S_1, \dots, S_{\ell}$ be these sliding actions in the order when they are performed. From the assumption, we know that $\ell > k$. 
By the pigeonhole principle, 
there exists some $U_i$ on which there are at least two sliding actions. Let $S_j$ and $S_{j'}$,  $j<j'$, be two sliding actions performed along the edges $ab$ and $a'b'$ of $U_i$ respectively,  where $b$ and $b'$ are initially  contaminated. So before the searcher on $a$ slides from $a$ to $b$, the vertex $a$ has two contaminated neighbors, $b$ and $b'$. This means that $S_j$ cannot be carried out by the definition of sliding. This is a contradiction, and hence $\ms(G) \geq |V(G)| - k$. Therefore $\ms(G) = |V(G)| - k$. 
\end{proof}

From Definition~\ref{def:k-clique-dismantlable}, it is not hard to see that in any clique-constructable graph, every cycle of length at least four has a chord. Thus clique-constructable graphs form a subclass of chordal graphs. We conclude this section by posing the following problem for chordal graphs more generally.
\begin{quest}
If $G$ is a chordal graph, what is $\zf(G)$?
\end{quest}

\section{Complexity} \label{Sec:Complexity}

For a graph $G$, a subset of vertices $V' \subset V(G)$ is a \dword{vertex cover} of $G$ if every edge of $G$ is incident to at least one vertex in $V'$.

\begin{thm}  \label{thm:NP1}
Given a connected graph $G$ and an integer $k$, the problem of determining whether $\zf(G) \leq k$ is NP-complete. This problem remains NP-complete for graphs with
maximum vertex degree 19.
\end{thm} 

\begin{proof}
It is easy to verify that the problem is in NP. To show that the problem is NP-hard, we will construct a reduction from the vertex cover problem for cubic graphs, which is NP-complete \cite{GJ79}. 
  
Let $(G, \ell)$ be an instance of the vertex cover problem, where $G$ is a connected cubic graph and $\ell$ is a positive integer. 
We will construct a graph $H$ and a positive integer $k$ that form an instance of the problem of deciding  whether $\ms(G) \leq k$. Let $G \boxtimes K_2$ be the strong product of $G$ and a complete graph with 2 vertices. 
Let $G_1$ and $G_2$ be the two copies of $G$ in $G \boxtimes K_2$. The vertices of $G \boxtimes K_2$ will be called \emph{base vertices} and
the edges in $G \boxtimes K_2$ that are copies of $K_2$ are called \emph{base edges}.
For each non-base edge $uv$ in $G \boxtimes K_2$, add two vertices $u'$, $v'$ and edges $u'v'$, $uu'$, $uv'$, $vv'$ and $vu'$. Let the new graph be denoted by $H$. 
Note that the vertices $u$, $v$, $u'$, $v'$ induce a complete graph with 4 vertices, which is called a \emph{$K_4$-block}. 
For each $K_4$-block in $H$, the two non-base vertices $u'$ and $v'$ are called \emph{inner vertices} of the $K_4$-block.
We finish the construction by setting $k = 4|E(G)| + |V(G)| + \ell$. 
It is easy to see that the graph $H$ can be constructed in polynomial time and the maximum vertex degree in $H$ is 19.
We will show that $G$ has a vertex cover of size at most $\ell$ if and only if $\ms(H) \leq k$.

 \begin{figure}[!htb]
		\begin{center}
	\begin{tikzpicture}[x=1cm,y=1cm,scale=1]
	
\draw [->, line width = 4] (3.5,7) -- (3.5,6.5);

\draw [fill = black] (0,0) circle (3pt) node[below left]{$a_1$};
\draw [fill = black] (7,0) circle (3pt) node[below right]{$b_1$};
\draw [fill = black] (0,5) circle (3pt) node[above left]{$a_2$};
\draw [fill = black] (7,5) circle (3pt) node[above right]{$b_2$};

\draw [fill = black] (3.5,5.5) circle (3pt) node[above]{$q_4$};
\draw [fill = black] (3.5,4.5) circle (3pt) node[below]{$p_4$};

\draw [fill = black] (3.5,0.5) circle (3pt) node[above]{$q_1$};
\draw [fill = black] (3.5,-0.5) circle (3pt) node[below]{$p_1$};

\draw [fill = black] (2.14,2.22) circle (3pt) node[above]{$q_2$};
\draw [fill = black] (2.8,1.3) circle (3pt) node[below]{$p_2$};

\draw [fill = black] (4.86,2.22) circle (3pt) node[above]{$q_3$};
\draw [fill = black] (4.2,1.3) circle (3pt) node[below]{$p_3$};

\draw (0,0) -- (7,0) -- (7,5) -- (0,5) -- (0,0);

\draw (0,0) -- (3.5,0.5) -- (7,0);	
\draw (0,0) -- (3.5,-0.5) -- (7,0);	
\draw (3.5,0.5) -- (3.5,-0.5);	

\draw (0,5) -- (3.5,5.5) -- (7,5);	
\draw (0,5) -- (3.5,4.5) -- (7,5);	
\draw (3.5,5.5) -- (3.5,4.5);	
	
\draw (0,5) -- (7,0);
\draw (7,5) -- (0,0);
	
\draw (0,0) -- (2.14,2.22) -- (7,5);
\draw (0,0) -- (2.8,1.3) -- (7,5);

\draw (7,0) -- (4.86,2.22) -- (0,5);
\draw (7,0) -- (4.2,1.3) -- (0,5);

\draw (2.14,2.22) -- (2.8,1.3);
\draw (4.86,2.22) -- (4.2,1.3);

\draw [fill = black] (1,7.5) circle (3pt) node[below]{$a$};
\draw [fill = black] (6,7.5) circle (3pt) node[below]{$b$};
\draw (1,7.5) -- (6,7.5);

	\draw (-1,7.5) node{$G$};
	\draw (-1,2.5) node{$H$};
	
	\end{tikzpicture}

			\caption{An edge in $G$ and its corresponding part in $H$, where $a_1a_2$ and  $b_1b_2$ are base edges, $a_1, a_2, b_1, b_2$ are base vertices, and $p_i$ and $q_i$ are inner vertices.}
			\label{fig:NP1}
		\end{center}
\end{figure}
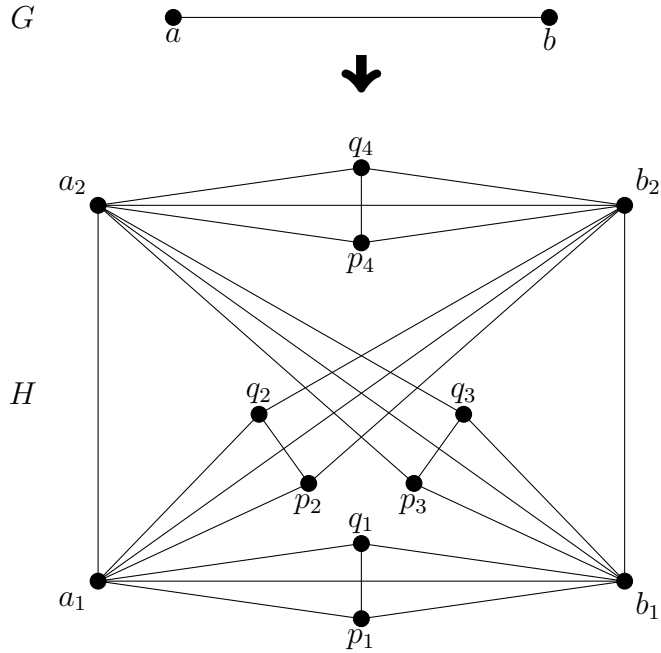

Initially, suppose that $U \subseteq V(G)$ is a vertex cover of $G$ of
size $\ell$.  Let $U_1 \subseteq V(G_1)$ and $U_2 \subseteq V(G_2)$ be the sets of vertices in $H$ that correspond to $U$.
Let $\overline{U} = V(G) \setminus U$, and
let $\overline{U}_1 \subseteq V(G_1)$ and $\overline{U}_2 \subseteq V(G_2)$ be the sets of vertices in $H$ that correspond to $\overline{U}$.
For each edge $ab$ in $G$, let $a_1, a_2$ correspond to $a$ and  $b_1, b_2$ correspond to $b$, where $a_1, b_1 \in V(G_1)$ and  $a_2, b_2 \in V(G_2)$.
Let the four $K_4$-blocks corresponding to $ab$ be induced by $\{a_1,b_1,p_1,q_1\}$, $\{a_1,b_2,p_2,q_2\}$, $\{a_2,b_1,p_3,q_3\}$ and $\{a_2,b_2,p_4,q_4\}$ respectively (see Figure~\ref{fig:NP1}).
We now describe a search strategy for $H$: Place a searcher on every vertex of $V(G_1)$ and $U_2$, and we also place a searcher on one of the inner vertices, say $p_i$,  in every $K_4$-block. 
Since $U$ is a vertex cover of $G$, for every edge $ab$ in $G$, at least one endpoint belongs to $U$.
For each pair of adjacent vertices $a, b \in U$, since the corresponding vertices $a_1, b_1 \in U_1$ and  $a_2, b_2 \in U_2$ are occupied, 
we can clear all $K_4$-blocks corresponding to $ab$  by sliding the searcher on $p_i$ to $q_i$, $1 \leq i \leq 4$. 
For each pair of adjacent vertices $a \in \overline{U}$ and $b \in U$, since the corresponding vertices $a_1 \in \overline{U}_1$ and  $b_1, b_2 \in U_2$ are occupied, 
we can clear the $K_4$-blocks by sliding searchers from $p_1$ to $q_1$, $p_2$ to $q_2$, $a_1$ to $a_2$, $p_3$ to $q_3$, and $p_4$ to $q_4$ (again refer to Figure~\ref{fig:NP1}). 
Thus $H$ can be cleared by $4|E(G)| + |V(G)| + \ell$ searchers. Hence  $\ms(H) \leq k$.

Conversely, suppose that $\ms(H) = k$. Consider a normalized optimal $\ms$-strategy $S$ for $H$. We will convert $S$ into an optimal $\ms$-strategy for $H$ such that each $K_4$-block is cleared by three searchers with a sliding action from one of the two inner vertices to the other. It is easy to see that every $K_4$-block is cleared by at least three searchers and at most four searchers. 

Suppose that there is a $K_4$-block in $H$ which is cleared by four searchers. Since every vertex can contain at most one searcher at any moment in an $\ms$-strategy, this $K_4$-block must be cleared by four searchers occupying the four vertices. So two searchers must be placed on the two inner vertices. If we place only one searcher on one inner vertex and let this searcher slide to the other  inner vertex, we can reduce the search number by one,  a contradiction. 

If there is a $K_4$-block in $H$ induced by $u$, $v$, $u'$, $v'$, where $u'$ and $v'$ are inner vertices, which is cleared by three searchers such that $u'v'$ is cleared by two searchers occupying $u'$ and $v'$, then there must be an edge in this $K_4$-block which is cleared by a sliding action. Without loss of generality, suppose that $uv$ is cleared by sliding a searcher from $u$ to $v$.
So the two searchers on $u'$ and $v'$ must be placed on them. We can change the action ``placing a searcher on $v'$'' in $S$ to ``placing a searcher on $v$'', and change the action ``sliding a searcher from $u$ to $v$'' in $S$ to ``sliding a searcher from $u'$ to $v'$''. It is easy to see that the new strategy is still an optimal $\ms$-strategy for $H$. After changing strategies for  $K_4$-blocks in this way, we can obtain an optimal $\ms$-strategy $S'$ for $H$ such that each $K_4$-block is cleared by three searchers with a sliding action from one of its inner vertices to the other. 

For each edge $ab$ in $G$, as illustrated in Figure~\ref{fig:NP1}, let $a_1a_2$ and  $b_1b_2$ be two base edges in $H$, where $a_1, b_1 \in V(G_1)$ and  $a_2, b_2 \in V(G_2)$.
In the optimal $\ms$-strategy $S'$, we place a searcher on each $p_i$, $1 \leq i \leq 4$. In total, we place $4|E(G)|$ searchers on inner vertices of $K_4$-blocks in $H$. The remaining $k - 4|E(G)|$ searchers are placed on base vertices of $H$. 
Note that the subgraph induced by $\{a_1, a_2, b_1, b_2\}$ is a $K_4$ (but it is not a $K_4$-block). So it is cleared by at least three searchers and at most four searchers. If the $K_4$ induced by $\{a_1, a_2, b_1, b_2\}$  is cleared by three searchers, since the edges $p_2q_2$ and $p_3q_3$ are cleared by sliding, the edges $a_1b_2$ and $a_2b_1$ cannot be both cleared by sliding. Thus one of $a_1a_2$ and $b_1b_2$ is cleared by a sliding and the other is cleared by two searchers on its two endpoints. Hence, for each edge $ab$ in $G$, at least one the two base edges $a_1a_2$ and  $b_1b_2$ is cleared by two searchers; let $\ell$ be the number of this kind of base edges in $H$. Thus $\ell = k - 4|E(G)| - |V(G)|$. 

We can then convert $S'$ into an optimal $\ms$-strategy for $H$ using $4|E(G)| + |V(G)| + \ell$ searchers in the following way: placing $4|E(G)|$ searchers on all inner vertices $p_i$; placing $|V(G)|$ searchers on all vertices of $G_1$;  placing $\ell$ searchers on a subset of vertices $V' \subset V(G_2)$ such that each base edge incident on one of these vertices is cleared by two searchers; the following sliding actions have been partially described above. Notice that  for every edge $a_2b_2$ in $G_2$, at least one of the endpoints is occupied by a searcher. Thus $V'$ is a vertex cover of $G_2$, and therefore, $G$ has a vertex cover of size at most $\ell$.
\end{proof}

\section{Algorithms} \label{Sec:Algorithm}

Given the result of Theorem~\ref{thm:NP1}, it is natural to ask if there are classes of graphs for which $\zf(G)$ can be computed efficiently. We begin by considering the problem on trees.  An algorithm to compute $\zf(T)$ for a tree $T$ was given in~\cite{BDF}, but its complexity was not considered.  In the next theorem, we describe an implementation of the Pruning Algorithm from~\cite{BDF}, recast in the language of pendent-edge dismantling.

\begin{thm}  \label{thm:tree}
If $T$ is a tree, then $\zf(T)$ and an optimal search strategy can be computed in linear time. 
\end{thm} 

\begin{proof} We again consider this as a constrained fast mixed search problem.
If $|E(T)|=0$, the problem is trivial, 
so suppose $T$ contains at least one edge. 
Note that $T$ contains at least two vertices of degree one. Furthermore, $T$ contains at least one pendent edge $u_1v_1$ with $\degree_{T}(v_1)=1$ 
such that at most one component in $T -  \{u_1, v_1\}$ contains edges.
The greedy algorithm can be used to find all such edges $\{u_1v_1, \dots, u_kv_k\}$ repeatedly until $T -  \{u_1, v_1, \dots, u_k, v_k\}$ is a null graph, where $u_iv_i$ is a pendent edge in $ T - \{u_1, v_1, \dots, u_{i-1}, v_{i-1}\}$ such that only one component in $T - \{u_1, v_1, \dots, u_{i-1}, v_{i-1}, u_i, v_i\}$ contains edges and all other components contain an isolated vertex each.

We can implement this greedy algorithm using a postorder traversal of $T$ to find the edges $\{u_1v_1, \dots, u_kv_k\}$ in the following way: first, arbitrarily pick a vertex as the root of $T$; every time  we visit a vertex $u$, if there is a child $v$ of $u$ that has not been paired with another vertex, then we pair $u$ with $v$. All the paired vertices will form the set $\{u_1v_1, \dots, u_kv_k\}$.
Then $(u_1 v_1, \ldots, u_k v_k)$ is a pendent-edge dismantling ordering of $T$.  
It follows from Theorem~\ref{thm:A-dismantlable} that $\ms(T) = n-k$.

We can modify this postorder traversal algorithm to construct an optimal search strategy of $T$ in constrained fast-mixed search as follows: every time when we pair a child $v_i$ with its parent $u_i$, we place a searcher on $v_i$. At the end of the algorithm, we place a searcher on each unpaired vertex. For $i$ from 1 to $k$, we slide the searcher on $v_i$ to $u_i$. After the $k$ sliding actions, $T$ will be cleared.
Since we use $n-k$ searchers, this is an optimal search strategy for $T$.

Note that the set $\{u_1v_1, \dots, u_kv_k\}$ can be found in linear time by the postorder traversal algorithm. Thus $\ms(T)$ and the above optimal search strategy can be computed in linear time.
\end{proof}

Indeed, the procedure described in Theorem~\ref{thm:tree} can be extended to compute $\zf(G)$ for any pendent-edge dismantlable graph $G$.  

\begin{cor} 
If $G$ is pendent-edge dismantlable, then $\zf(G)$ and an optimal search strategy can be computed in polynomial time.  
\end{cor} 

By successively deleting pendent edges, we can also compute $\zf(G)$ on the class of unicyclic graphs in linear time.

\begin{cor} 
If $G$ is a unicyclic graph, then $\zf(G)$ and an optimal search strategy can be computed in linear time. 
\end{cor} 

\begin{proof}
A unicyclic graph is a connected graph containing exactly one cycle. Let $C$ be the cycle in $G$. Again, we consider the constrained fast mixed search model.
Let $H$ be a graph obtained from $G$ by deleting pendent edges and their endpoints recursively. Suppose $k$ pendent edges and their endpoints are deleted. 
If $H$ is a null graph, then $\ms(G) = |V(H)| + k$. 
If $H$ contains at least one edge, then either $H=C$ or $C$ is a component in $H$ and all other components are isolated vertices. Theorem~4.2 of~\cite{BDF} asserts that $\ms(C) = \left\lceil\frac{|V(C)|}{2}\right\rceil$; hence $\ms(H) =  \left\lceil \frac{|V(C)|}{2} \right\rceil + |V(G)| - 2k - |V(C)| =  |V(G)| - \left\lfloor \frac{|V(C)|}{2} \right\rfloor - 2k $.
Thus, from Corollary~\ref{cor:G-leaves}, we have $\ms(G) = \ms(H) + k = |V(G)| - \left\lfloor \frac{|V(C)|}{2} \right\rfloor - k$.
Similar to the postorder traversal algorithm in the proof of Theorem~\ref{thm:tree}, we can find $\ms(G)$ and an optimal search strategy in linear time.
\end{proof}

Another application of pendent edge removal is the following result.  Despite the NP-completeness result for computing $\zf(G)$ on an arbitrary graph $G$, in certain cases we can efficiently compute $\zf(H)$ for a graph $H$ containing $G$ as a subgraph.  In particular, if we form $H$ by attaching a tree to each vertex of $G$ in such a way that we can delete the vertices of $G$ by successive deletion of pendent edges, then Corollary~\ref{cor:G-leaves} tells us the value of $\zf(H)$. We thus obtain the following theorem.

\begin{thm}  \label{thm:attach-trees}
Let $G$ be a graph. Let $H$ be a graph that can be obtained from $G$ by attaching a tree to each vertex of $G$. For each vertex $v$ of $G$, let $T_v$ be the tree attached to $v$. If for each $v$ of $G$, there is a leaf of $T_v$ such that the distance between this leaf and $v$ is odd. Then $\zf(H)$ and an optimal search strategy can be computed in linear time. 
\end{thm}

\section{Algorithms on Cacti}  \label{Sec:Cacti}

A \dword{cactus} is a connected graph such that any two cycles can share at most one vertex. 
A \dword{cactus forest} is a graph where every connected component is a cactus.
Let $G$ be a cactus forest.
A vertex of $G$ is called a \dword{leaf vertex} if it is of degree 1. 
A cycle of $G$ is called a \dword{leaf cycle} if it contains exactly one vertex of degree $\geq 3$ and
this vertex is called an \dword{extender} of the leaf cycle. 
If a connected component in $G$ is a cycle, then we say that this cycle is \dword{isolated}.

In this section, we consider a slightly general version of the constrained fast-mixed searching: some vertices in the graph $G$ are pre-occupied by searchers initially; these searchers are called \dword{ pre-occupying} searchers. The pre-occupying searchers have not moved to their neighbors at the beginning of the searching process, and any subset of them may slide to neighbors during the searching process.
The goal is to find the minimum number of additional searchers (excluding the pre-occupying searchers) to clear $G$, and for convenience,  this number is still denoted by $\ms(G)$.

In our algorithm, we reduce the size of the input graph in every step. 
When we perform one step, we assume that all previous steps are not applicable. 
In each step of the algorithm, we will clear some vertices and delete some cleared vertices from the graph.
Our algorithm contains the following five steps.

\textbf{Step 1:} 
If there is a degree 1 contaminated vertex $v$ adjacent to a contaminated neighbor $u$, 
then place a searcher on $v$ and slide it to $u$. Delete both of $v$ and $u$ from the graph.

From the proof of Theorem~\ref{thm:tree}, we know that Step 1 is correct.

\textbf{Step 2:} 
If there is a degree 1 pre-occupied vertex $v$ with a pre-occupied neighbor $u$, then delete $v$ from the graph.
(Since the searcher on $v$ cannot slide to its neighbor, we can delete $v$ from the graph safely.)

\textbf{Step 3:} 
If there is a degree 1 pre-occupied vertex $v$ with a contaminated neighbor $u$, then slide the pre-occupying searcher from $v$ to $u$. Delete both of $v$ and $u$ from the graph.

\begin{lemma} \label{lem:Step3}
If $H$ is a graph that has a degree 1 pre-occupied vertex $v$ whose neighbor $u$ is contaminated, then there is an optimal search strategy for $H$ in which the pre-occupying searcher on $v$ slides from $v$ to $u$. 
\end{lemma}

\begin{proof}
Let $S$ be an optimal search strategy for $H$. If the pre-occupying searcher on $v$ does not slide from $v$ to $u$, there are two possible cases: 
if there is a searcher sliding to $u$ from another vertex in $S$, we can simply let the pre-occupying searcher slide from $v$ to $u$ to get another optimal search strategy. 
Otherwise, if a searcher is placed on $u$ in $S$, then this search must slide to a vertex $u'\neq v$; otherwise we can reduce the search number by one, which is a contradiction.
In this case, we can obtain another optimal search strategy by simply placing a searcher on $u'$ and sliding the pre-occupying searcher from $v$ to $u$.
From these two cases, we know that there is an optimal search strategy for $H$ in which the pre-occupying searcher on $v$ slides from $v$ to $u$.
\end{proof}

\textbf{Step 4:} If there is a degree 1 contaminated vertex $v$ adjacent to a pre-occupied neighbor $u$, then slide the searcher on $u$ from $u$ to $v$. Delete both of $v$ and $u$ from the graph.

\begin{lemma} \label{lem:Step4}
If $H$ is a graph that has a degree 1 contaminated vertex $v$ whose neighbor $u$ is pre-occupied, then there is an optimal search strategy for $H$ in which the pre-occupying searcher on $u$ slides from $u$ to $v$. 
\end{lemma}

\begin{proof}
Let $S$ be an optimal search strategy for $H$. If the pre-occupying searcher on $u$ does not slide from $u$ to $v$, 
then we must place a searcher on $v$ in $S$. 
Furthermore, it is easy to see that the searcher on $u$ must slide to another neighbor $v' \not=v$; otherwise the searchers used by $S$ can be reduced by one, which is a contradiction. 
For the action of sliding the pre-occupying searcher from $u$ to $v'$ in $S$, we can replace it with two actions: placing a searcher on $v'$ and sliding the searcher on $u$ from $u$ to $v$. We also need to delete the action of placing a searcher on $v$ in $S$.
Thus, we obtain another strategy without increasing the number of searchers. Moreover, the action of sliding the searcher from $u$ to $v$ can be delayed since the contaminated vertex $v$ is  adjacent only to vertex $u$.
Hence there is an optimal search strategy for $H$ in which the pre-occupying searcher on $u$ slides from $u$ to $v$.
\end{proof}

\begin{lemma} \label{lem:Steps1-4}
After Steps 1 -- 4, if the remaining graph $H$ is a cactus forest that contains at least one edge, then $H$ contains either an isolated cycle or a leaf cycle. 
\end{lemma}

\begin{proof}
Note that in the algorithm, when one step is performed, all previous steps are not applicable. Thus there is no degree 1 vertex in $H$. Hence $H$ contains  at least one cycle. Notice that $H$ is a cactus forest. If there is no isolated cycle in $H$,  then every component of $H$ must contain at least two leaf cycles.
\end{proof}

\textbf{Step 5:} 
We first remove all pre-occupied isolated vertices. For each unoccupied isolated vertex, we place a searcher on it and then delete it.
For each isolated cycle, we clear it using the method in the proof of Lemma~\ref{lem:Step5}.
Let $H$ be the current cactus forest that does not contain any isolated vertex, isolated cycle, or degree 1 vertex.
Let $C$ be a leaf cycle in $H$, where some vertices of $C$ may be pre-occupied, and let $v$ be the extender of $C$. 
Let $P$ be the path obtained from $C$ by deleting the extender $v$.
\begin{enumerate}
\item[5.1.]  
If $v$ is pre-occupied and $\ms(P) =  \ms(C)$, then clear $P$ using an optimal search strategy. Delete all vertices of $P$ from the graph.
\item[5.2.] 
If $v$ is pre-occupied and $\ms(P) =  \ms(C)+1$, then clear $C$ using an optimal search strategy. Delete all vertices of $C$ from the graph.
\item[5.3.] 
If $v$ is contaminated and $\ms(P) =  \ms(C) -1$, then clear $P$ using an optimal search strategy. Delete all vertices of $P$ from the graph.
\item[5.4.] 
If $v$ is contaminated and $\ms(P) =  \ms(C)$, then clear $C$ using an optimal search strategy. Delete all vertices of $C$ from the graph.
\end{enumerate}

\begin{lemma} \label{lem:Step5}
Let $H$ be a cactus forest that does not have any degree 1 vertices. 
Suppose that $H$ has a leaf cycle $C$ whose extender is $v$. 
Let $P$ be the path obtained from $C$ by deleting the extender $v$.
Then there is an optimal search strategy for $H$ where $C$ or $P$ is cleared using one of the strategies  
in Steps 5.1 -- 5.4. 
\end{lemma}

\begin{proof}
Note that $\ms(C)$ (or $\ms(P)$) is the minimum number of searchers, excluding the pre-occupying searchers, to clear $C$ (or $P$).
When $v$ is pre-occupied, since any solution to $P$ is also a solution to $C$, we have $\ms(C) \leq  \ms(P) \leq \ms(C)+1$.
When $v$ is contaminated, we have $\ms(P) \leq  \ms(C) \leq \ms(P)+1$ because for any optimal search strategy $S_1$ of $P$, we can get a search strategy $S_2$ for $C$ by adding the action of placing a searcher on $v$ at the begin of $S_1$. Thus, no case is missing in Step 5.

Note that $H$ does not contain any degree 1 vertices. 
Let $S$ be an optimal search strategy for $H$, and let $X$ be the set of actions in $S$ that are related to vertices of cycle $C$.
\medskip

{\sc Case 1.} $v$ is pre-occupied in $H$, and in $S$ the searcher on $v$ does not move. 

In this case,  it is easy to see that $\ms(P) =  \ms(C)$. 
So we replace $X$ with an optimal search strategy for $P$ and delete all vertices of $P$ from $H$ (Step 5.1).
\medskip

{\sc Case 2.} $v$ is pre-occupied in $H$, and in $S$ the searcher on $v$ slides from $v$ to a vertex of $C$. 

Note that in this case, some searchers may move to neighbors of $v$ (not in $C$) first and then the searcher on $v$ slides from $v$ to a vertex of $C$. We can assume that all actions in $X$ are the last actions in $S$.
 If $\ms(P) =  \ms(C)$, then we replace $X$ by an optimal search strategy for $P$ and delete all vertices of $P$ from $H$ (Step 5.1).
If $\ms(P) =  \ms(C)+1$, then we replace $X$ by an optimal search strategy for $C$ and delete all vertices of $C$ from $H$ (Step 5.2).
\medskip

{\sc Case 3.} $v$ is pre-occupied in $H$, and in $S$ the searcher on $v$ slides from $v$ to a vertex outside $C$. 

Notice that some searchers may move to neighbors of $v$ (in $C$) first and then the searcher on $v$ slides from $v$ to a vertex outside $C$. We can assume that all actions of $X$ are the earliest actions in $S$. For this case, we have $\ms(P) =  \ms(C)$, and so 
we replace $X$ by an optimal search strategy for $P$ and delete all vertices of $P$ from $H$ (Step 5.1).
\medskip

{\sc Case 4.} $v$ is a contaminated vertex in $H$, and in $S$ a searcher on a vertex of $C$ slides to $v$.

In this case we have $\ms(P) =  \ms(C)$. So  
 we replace $X$ by an optimal search strategy for $C$ and delete all vertices of $C$ from $H$ (Step 5.4).
\medskip

{\sc Case 5.} $v$ is a contaminated vertex in $H$, and in $S$ the searcher on a vertex $u$ outside $C$ slides from $u$ to $v$. 

Let $X'$ denote the set of actions obtained from $X$ by deleting the action that the searcher on $u$ slides from $u$ to $v$. 
If $\ms(P) =  \ms(C) -1$, since any optimal search strategy for $P$ is a part of an optimal search strategy for $C$,
we replace $X'$ by an optimal search strategy for $P$ and delete all vertices of $P$ from $H$ (Step 5.3).
If $\ms(P) =  \ms(C)$, let $S_C$ be an optimal search strategy for $C$, and 
we can construct a new search strategy $S'_C$ for $C$ as follows:
if a searcher is placed on $v$ in $S_C$ and this searcher slides from $v$ to a vertex $w$ in $C$, then in $S'_C$ we replace this sliding action by a placing action in which we place a searcher on $w$; otherwise, we let $S'_C=S_C$. Since the number of new searchers used to clear $C$ in $S$ is at least $\ms(P)$, we can replace $X$ by $S'_C$ without increasing the number of  new searchers (Step 5.4). 
\medskip

{\sc Case 6.} $v$ is a contaminated vertex in $H$, a searcher is placed on $v$ in $S$, and either the searcher on $v$ slides to a vertex $u$ in $C$ or the searcher on $v$ does not move.

If $\ms(P) =  \ms(C)$,  we replace $X$ by an optimal search strategy for $C$ and delete all vertices of $C$ from $H$ (Step 5.4).
 If $\ms(P) = \ms(C)-1$, we replace $X$ by an optimal search strategy for $P$ and delete all vertices of $P$ from $H$ (Step 5.3).
\medskip

{\sc Case 7.} $v$ is a contaminated vertex in $H$, a searcher is placed on $v$ in $S$, and the searcher on $v$ slides to a vertex $u$ outside $C$.

If $\ms(P) =  \ms(C)-1$, let $X'$ denote the set of actions obtained from $X$ by deleting two actions: ``placing a searcher on $v$'' and ``the searcher on $v$ slides from $v$ to $u$''. 
Then we replace $X'$ by an optimal search strategy for $P$ and delete all vertices of $P$ from $H$ (Step 5.3).
If $\ms(P) =  \ms(C)$, 
then we replace $X$ by an optimal search strategy for $C$ and add the action ``placing a searcher on $u$ ''. Then delete all vertices of $C$ from $H$ (Step 5.4).

From the above cases, we know that  there is an optimal search strategy for $H$ where $C$ or $P$ is cleared using one of the strategies  
in Steps 5.1 -- 5.4. 
\end{proof}

\begin{thm}
If $G$ is a cactus, then $\zf(G)$  and an optimal search strategy for $G$ can be computed  in linear time.
\end{thm}

\begin{proof}
We first show that $\ms(G)$  and an optimal search strategy for $G$ can be computed by Steps 1 -- 5.
It is easy to see that every time when we run Step 1, the number of searchers is increased by one, and when we run Steps 2 -- 4, the number of searchers does not change. 
From Lemmas~\ref{lem:Step3} and \ref{lem:Step4}, there is an optimal search strategy for $G$ that contains the strategies used in Steps 1 -- 4.
We now consider Step 5. 
In Step 5, we need to compute $\ms(P)$ and $\ms(C)$, where $P$ and $C$ may contain pre-occupied vertices. From the above argument, $\ms(P)$ and an optimal search strategy for $P$ can be computed by Steps 1 -- 4. 
When we compute $\ms(C)$, there are three cases.
\begin{enumerate}
\item
$C$ does not contain any pre-occupied vertices. In this case, it follows from Theorem~4.2 of \cite{BDF}  that $\ms(C)=\lceil |C|/2 \rceil.$
\item
$C$ contains two pre-occupied vertices $u$ and $v$ that are adjacent. Let $P'$ be the path obtained from $C$ by deleting the edge $uv$ from $C$. It is easy to see that 
$\ms(P') =  \ms(C)$, and furthermore, an optimal search strategy for $P'$ can be computed by Steps 1 -- 4 in $O(|P'|)$ time, which is also  an optimal search strategy for $C$.
\item
$C$ contains pre-occupied vertices and none of them are adjacent. 
Let $u$ be a pre-occupied vertex on $C$ and let $v$ and $v'$ be the neighbors of $u$ on $C$. For an optimal search strategy for $C$, there are three possible subcases related to $u$: 
\begin{enumerate}
\item
The pre-occupying searcher on $u$ does not move. Let $P'$ be the path obtained from $C$ by deleting $u$. In this case, we have $\ms(C) = \ms(P')$, where $\ms(P')$ can be computed by Steps 1 -- 4. 
\item 
The pre-occupying searcher on $u$ slides from $u$ to $v$ (resp. $v'$). Let $P'$ be the path obtained from $C$ by deleting $u$ and $v$ (resp. $v'$). In this case, we have $\ms(C) = \ms(P')$, where $\ms(P')$ can be computed by Steps 1 -- 4. 
\item
A searcher is placed on $v$ (resp. $v'$). Let $P'$ be the path obtained from $C$ by deleting the edge $uv$ (resp. $uv'$). In this case, we have $\ms(C) = \ms(P')$, where $\ms(P')$ can be computed by Steps 1 -- 4. 
\end{enumerate}

\end{enumerate}
Note that $\ms(C)$ is the minimum among the above three subcases.

It follows from Lemma~\ref{lem:Step5} that there is an optimal search strategy for $G$ that contains the strategies used in Step 5.
From Lemma~\ref{lem:Steps1-4} and Step 5, we know that following Steps 1 -- 5, the graph will be reduced to an empty graph and the strategy is optimal.

Next, we show that the algorithm (Steps 1 -- 5) can be implemented in linear time. 
Suppose the input graph is represented by an adjacency list. Note that in $G$, every edge is contained in at most one cycle. 
We can use Depth-First Search (DFS) to find all cycles in linear time. Then we construct an auxiliary tree $T$ as follows: 
for each vertex of $G$ that is not contained in any cycle, it is represented by a node in $T$, called a noncycle-node; for each cycle $C$ in $G$, it is represented by a node in $T$, called a cycle-node, whose degree in $T$ is the number of vertices of $C$ with degree more than 2 in $G$; for each vertex of $G$ that is contained in a cycle with degree more than two in $G$, it is  represented by a node in $T$ which is linked to the corresponding cycle-nodes and noncycle-nodes; we ignore each degree 2 vertex of $G$ that is contained in a cycle.

We can pick a cycle-node as the root of $T$ and use DFS  to traverse and modify $T$: every time a leaf of $T$ is visited, if it is a noncycle-node, then run one of Steps 1 -- 4 and modify $T$ accordingly; otherwise, run Step 5 and modify $T$ accordingly. 
Note that each modification of $T$ can be done in $O(1)$. For a cycle with $k$ vertices, it takes $O(k)$ time to run Step 5 as shown in the above.
Since no edge of $G$ is contained in more than one cycle, $\ms(G)$  and an optimal search strategy for $G$ can be computed  in linear time.
\end{proof}

\section{Discussion} \label{Sec:Conclusion}

As we had noted in Section~5, chordal graphs naturally suggest themselves do their relation to clique dismantable graphs. One family of chordal graphs in particular that seems natural to consider is split graphs. Since the vertex sets of split graphs can be partitioned into an independent set and a set inducing a clique, we can see that the results of Theorem~\ref{thm:LB-max-indep} and of \cite{BDF} give lower bounds. Further, since split graphs have shown themselves to be an interesting family computationally (having polynomial-time algorithms for known NP-complete problems on arbitrary graphs), they would be a sensible starting point.

In~\cite[Theorem~6.1.1]{thesis}, it is proved that if $G$ is formed by identifying vertices from disjoint graphs $G_1$ and $G_2$, then $\zf(G) \leq \zf(G_1) + \zf(G_2)$.  This procedure may be viewed as contracting a bridge in a graph; that is, if $H$ is a graph with a bridge $e$ and $H^*$ is formed from $H$ by contracting $e$, then $\zf(H^*) \leq \zf(H)$.  This leads to a natural open question: 
\begin{quest} \label{ContractionQuestion}
If $G$ is a graph and $G^*$ is the contraction of $G$ at an edge $e=xy$, is $\zf(G^*) \leq \zf(G)$?  
\end{quest}
It is not difficult to prove that $\zf(G^*) \leq \zf(G)+1$; letting $\infty$ be the vertex formed by identification of $x$ and $y$, if $x$ or $y$ is occupied place a cop in $G^*$ on $\infty$ and the target(s) of $x$ and/or $y$ in $G$. However, we have thus far been unable to find an example of a graph for which $\zf(G^*) > \zf(G)$.

If the answer to Question~\ref{ContractionQuestion} is ``yes'', then $\zf(G)$ would be bounded below by $\zf(H)$ for any induced minor of $G$.  In particular, the existence of a clique as a minor would provide a lower bound of $\zf(G)$. 
This suggests considering the constrained zero forcing number of planar graphs. Planar graphs have a strong place in classic cops and robber literature, so these would again be a natural family to consider. 

Finally, we recognize that our constrained zero forcing might be better called ``1-constrained zero forcing,'' since every vertex can force at most one other vertex to be colored, and none of the newly colored vertices can force a vertex. However, there is a straightforward generalization here to $k$-constrained zero forcing, where each initially colored vertex can force a vertex, then the newly forced vertex could force a vertex, and that vertex could force a vertex, and so on, up to $k$ times. This is more easily considered, perhaps, in the deduction game or in constrained fast mixed search, where $k$-deduction would allow a searcher to fire or slide up to $k$ times. In particular, for some large enough $k$, this becomes exactly the zero forcing process. However, characterizing which $k$ would be sufficient would be similar to the determining the capture time of a game of cops and robber.

\end{document}